\documentclass[reqno,12pt]{amsart}
\usepackage{amsmath,amssymb,amsthm,graphicx,mathrsfs,url}
\usepackage[usenames,dvipsnames]{color}
\usepackage[colorlinks=true,linkcolor=Red,citecolor=Green]{hyperref}
\usepackage[super]{nth}
\usepackage[open, openlevel=2, depth=3, atend]{bookmark}
\hypersetup{pdfstartview=XYZ}
\usepackage[font=footnotesize]{caption}
\usepackage{a4wide}

\newtheorem{thm}{Theorem}
\newtheorem{prop}{Proposition}[section]

\newtheorem{lem}[prop]{Lemma}
\newtheorem{cor}[prop]{Corollary}
\theoremstyle{definition}
\newtheorem{defi}[prop]{Definition}
\newtheorem{rem}[prop]{Remark}
\numberwithin{equation}{section}

\newcommand{\mc}{\mathcal}
\newcommand{\rr}{\mathbb{R}}
\newcommand{\nn}{\mathbb{N}}
\newcommand{\cc}{\mathbb{C}}

\newcommand{\zz}{\mathbb{Z}}

\newcommand{\la}{\lambda}
\newcommand{\eps}{\epsilon}

\newcommand{\pl}{\partial}
\newcommand{\x}{\times}

\newcommand{\til}{\widetilde}
\newcommand{\bbar}{\overline}

\newcommand{\cjd}{\rangle}
\newcommand{\cjg}{\langle}

\newcommand{\demi}{\tfrac{1}{2}}

\newcommand{\energy}{E}
\newcommand{\length}{L}

\DeclareMathOperator{\supp}{supp}
\DeclareMathOperator{\Vol}{Vol}

\newcommand{\N}{\mathbb{N}}
\newcommand{\Z}{\mathbb{Z}}
\newcommand{\R}{\mathbb{R}}

\newcommand{\nnabla}{\nabla\!}

\newcommand{\cN}{\mathcal{N}}

\newcommand{\nablav}{\nabla^{\mathrm{v}}}
\newcommand{\nablah}{\nabla^{\mathrm{h}}}

\begin{document}
\title{Boundary and lens rigidity for non-convex manifolds}
\author{Colin Guillarmou}
\address{CNRS, Universit\'e Paris-Sud, D\'epartement de Math\'ematiques, 91400
Orsay, France}
\email{cguillar@math.cnrs.fr}

\author{Marco Mazzucchelli}
\address{CNRS, \'Ecole Normale Sup\'erieure de Lyon, UMPA, 69364 Lyon Cedex 07, France}
\email{marco.mazzucchelli@ens-lyon.fr}

\author{Leo Tzou}
\address{School of Mathematics and Statistics, University of Sydney, Sydney, Australia}
\email{leo@maths.usyd.edu.au}

\date{November 28, 2017}

\begin{abstract}
We study the boundary and lens rigidity problems on domains without assuming the convexity of the boundary. We show that such rigidities hold when the domain is a simply connected compact Riemannian surface without conjugate points. For the more general class of non-trapping compact Riemannian surfaces with no conjugate points, we show lens rigidity. We also prove the injectivity of the $X$-ray transform on tensors in a variety of settings with non-convex boundary and, in some situations, allowing a non-empty trapped set.
\end{abstract}

\maketitle

\section{Introduction}

We consider a compact Riemannian manifold with boundary $(M,g)$, and study the geometric inverse problems consisting in the determination of its geometry from boundary measurements (throughout the paper, $M$ will be tacitly assumed to be connected, unlike its boundary $\pl M$). As boundary data, we will employ the \emph{boundary distance function}
\begin{equation}\label{betag}
\beta_g:=d_g|_{\pl M\x \pl M},
\end{equation} 
where $d_g:M\times M\to[0,\infty)$ is the Riemannian distance, and the \emph{lens data}   
\[ \tau^+_g: \pl SM\to [0,\infty], \qquad \sigma_g: \pl SM\setminus \Gamma_-\to \pl SM.\] 
Here, $SM$ denotes the unit tangent bundle, $\Gamma_-:=\{y\in \pl SM\ |\ \tau_g^+(y)=+\infty\}$, the \emph{exit time} $\tau^+_g(x,v)$ is 
 the maximal non-negative time of existence of the geodesic $\gamma_{x,v}(t)=\exp_x(tv)$, and the \emph{scattering map}  $\sigma_g(x,v):=(\gamma_{x,v}(\tau_g^+(x,v)),\dot{\gamma}_{x,v}(\tau_g^+(x,v)))$
gives the exit position and ``angle'' of $\gamma_{x,v}$.
When $\tau^+_g$ is everywhere finite, $(M,g)$ is said to be \emph{non-trapping}. The \emph{boundary rigidity problem} asks whether the boundary distance $\beta_g$ determine $(M,g)$ up to diffeomorphisms fixing $\pl M$. Analogously, the \emph{lens rigidity problem} asks whether the lens data $(\tau_g^+,\sigma_g)$  determine $(M,g)$ up to diffeomorphisms fixing $\pl M$. For \emph{simple} Riemannian manifolds, that is, compact Riemannian balls with strictly convex boundary and without conjugate points, these two rigidity problems are equivalent, since the boundary distance and the lens data can be easily recovered from each another. Without the convexity assumption on $\partial M$, this equivalence becomes unclear, for the length minimizing curves are not necessarily geodesics anymore. 

Our first theorem provides a new boundary rigidity result.
\begin{thm}\label{th1}
Let $M$ be a simply connected compact surface with boundary. If $g_1$ and $g_2$ are two Riemannian metrics on $M$ without conjugate points such that 
 $\beta_{g_1}=\beta_{g_2}$, then there is a diffeomorphism $\psi:M\to M$  such that $\psi|_{\pl M}={\rm Id}$ and $\psi^*g_2=g_1$.
\end{thm}
The proof is carried out in two steps: we first prove that $(M,g_1)$ and $(M,g_2)$ are isometric provided they have the same lens data; then we show in Theorem~\ref{betadetS} that the boundary distance determines the lens data for a class of Riemannian metrics including those considered in Theorem~\ref{th1}. The proofs of these two facts turn out to be much more intricate than for simple Riemannian metrics due to the presence of glancing geodesics and to the discontinuities of the lens data. 
We remark that Theorem~\ref{betadetS} is likely false in higher dimension, due to the fact that there is an example due to Croke and Wen \cite{Croke:2015qy} of a simply connected 3-manifold  $(M,g)$ with boundary, without conjugate points and with a geodesic $\gamma$ of length $\ell>0$ with endpoints 
$x,x'\in\pl M$ that is not length minimizing, i.e.\ $\beta_{g}(x,x')<\ell$.

Given two compact Riemannian manifolds $(M_1,g_1)$ and $(M_2,g_2)$ with isometric boundaries, one can always find open neighborhoods $U_i\subset M_i$ of the boundaries $\partial M_i$ and a diffeomorphism $\phi:U_1\to U_2$ extending the isometry between $(\pl M_1,g_1|_{\pl M_1})$ and 
$(\pl M_2,g_2|_{\pl M_2})$ such that $\phi^*g_2|_{\pl M_1}=g_1|_{\pl M_1}$. We say that $(M_1,g_1)$ and $(M_2,g_2)$ agree to order $k\in \nn$ at the boundary when, for one such diffeomorphism $\phi$, the $k$-jets of $\phi^*g_2$ and $g_1$ coincide at all points of $\partial M_1$. We say that  $(M_1,g_1)$ and $(M_2,g_2)$ have the same scattering map if $d\phi^{-1}\circ \sigma_{g_2}\circ d\phi|_{\partial SM_1}= \sigma_{g_1}$, and that they have the same lens data if they further satisfy $\tau^+_{g_2}\circ d\phi|_{\partial SM_1}=\tau^+_{g_1}$. Our second theorem is a scattering rigidity result for non-trapping Riemannian surfaces. 
\begin{thm}\label{th2}
Let $(M_1,g_1)$ and $(M_2,g_2)$ be two non-trapping, oriented compact Riemannian surfaces with boundary, without conjugate points, that agree to order $2$ at the boundary.
If they have the same scattering map, then there exists a diffeomorphism $\psi:M_1\to M_2$ extending the isometry between $(\pl M_1,g_1|_{\pl M_1})$ and 
$(\pl M_2,g_2|_{\pl M_2})$ such that $\psi^*g_2=e^{\rho}g_1$ for some $\rho\in C^\infty(M_1)$ with $\rho|_{\partial M_1}\equiv0$.
\end{thm} 
In \cite{Stefanov:2009lp}, Stefanov and Uhlmann proved that  two non-trapping Riemannian metrics without conjugate points and with the same lens data agree to infinite order at the boundary. Therefore Theorem \ref{th2} shows that such metrics are conformally equivalent.

We emphasize that both Theorem \ref{th1} and  \ref{th2} are not a  consequence of 
known results for simple Riemannian manifolds. Indeed, there are manifolds $(M,g)$ satisfying the assumptions of Theorem \ref{th2} but that cannot be embedded isometrically in a simple Riemannian manifold: for example, cutting open just on the right of the unique closed geodesic in a negatively curved cylinder with convex boundary and keeping the right connected component, we obtain a non-trapping cylinder with negative curvature whose boundary has one concave component and one convex component (see Figure~\ref{f:cylinder}). There are self-intersecting geodesics in this cylinder, thus it cannot be embedded isometrically in some simple Riemannian manifold. In the setting of Theorem \ref{th1}, if the manifold $(M,g)$ is a domain of a simple manifold $(\til{M},\til{g})$, then the boundary distance $\beta_g$ is not directly related to the boundary distance $\beta_{\til{g}}$, and Theorem \ref{betadetS} would still be necessary to prove Theorem \ref{th1}. 
It also seems unlikely that for all $(M,g)$ as in Theorem \ref{th1} there is a simple 
extension $(\til{M},\til{g})$, and, even if that were the case, it seems to us a difficult problem to build such an extension when there are pair of points in $\pl M$ that are close to being conjugate points.

\begin{figure}
\begin{center}
\begin{footnotesize}
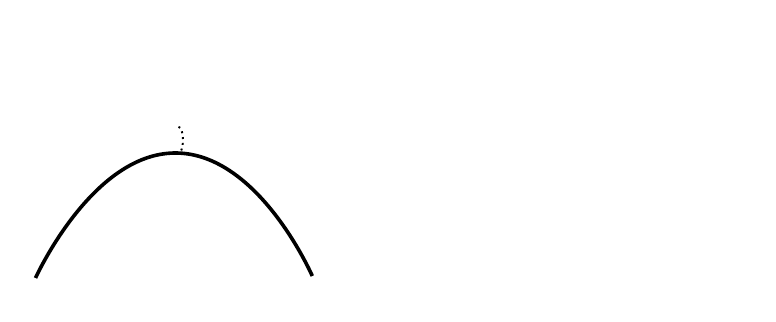 
\end{footnotesize}
\caption{\textbf{(a)} A negatively curved cylinder with convex boundary. \textbf{(b)} 
A compact subcylinder with a concave boundary component. The picture shows a self-intersecting geodesic.}
\label{f:cylinder}
\end{center}
\end{figure}%

In his Ph.D.\ thesis \cite{Zhou:2011wq}, Zhou showed how to derive the following lens rigidity result from Theorem~\ref{th2} together with a result of Croke \cite[Theorem 1.2]{Croke:2005bh}. We will sketch its proof in Section~\ref{ss:final}.
\begin{thm}\label{th2bis}
Let $(M_1,g_1)$ and $(M_2,g_2)$ be two  non-trapping, oriented compact Riemannian surfaces with boundary, without conjugate points, and with the same lens data. Then there exists a diffeomorphism $\psi:M_1\to M_2$ such that $\psi^*g_2=g_1$.
\end{thm} 
Theorem~\ref{th2bis}, together with a recent result of Croke and Wen \cite{Croke:2015qy}, also implies a new scattering rigidity result for surfaces, see Corollary~\ref{c:croke}. 

The linearization of the lens rigidity problem, which is strongly employed in the proofs of Theorems \ref{th1} and \ref{th2}, consists in studying the injectivity of the X-ray transform on symmetric $2$-tensors on $M$, or on functions on $M$ if the metric is varied within a conformal class. We denote by $\otimes_S^m T^*M$ the vector bundle of symmetric covariant $m$-tensors on $M$. The X-ray transform on such tensor fields is the operator $I_m: C^0(M;\otimes_S^m T^*M)\to L^\infty_{\rm loc}(\pl SM\setminus \Gamma_-)$ defined by
\begin{equation}
\label{e:I_m}
I_mf(x,v):= \int_0^{\tau^+_g(x,v)}f_{\gamma_{x,v}(t)}(\big.\dot\gamma_{x,v}(t),...,\dot\gamma_{x,v}(t))dt,
\end{equation}
A straightforward computation shows that $I_m(Dp)=0$ for each $p\in C^1(M;\otimes_S^{m-1}T^*M)$ with $p|_{\pl M}=0$, where $D$ denotes the symmetrized covariant derivative.  Our third result is the following.
\begin{thm}\label{th3}
Let $(M,g)$ be an oriented compact Riemannian surface with boundary that is non-trapping and without conjugate points. Let $f\in C^1(M; \otimes_S^mT^*M)\cap\ker(I_m)$ for some $m\geq 0$.  If $m=0$, then $f=0$. If $m\geq 1$, there exists $p\in C^2(M; \otimes_S^{m-1}T^*M)$ such that $p|_{\pl M}\equiv0$ and $f=Dp$.
\end{thm}

In a compact Riemannian manifold with boundary $(M,g)$, we define the trapped set $\mc{K}\subset SM$ to be the set of points $(x,v)\in SM$ such that the corresponding geodesic $\gamma_{x,v}(t)=\exp_x(tv)$ is defined for all $t\in\R$. Our last result is valid for compact manifolds of any dimension that are possibly trapping.

\begin{thm}\label{th4}
Let $(M,g)$ be a compact Riemannian manifold with boundary and with no conjugate points. If the trapped set $\mc{K}$ is a (possibly empty) hyperbolic set contained in the interior of $SM$, then:
\begin{itemize}
\item[(i)] If  $f\in C^1(M)\cap\ker(I_0)$, then $f=0$.

\item[(ii)] If $f\in C^2(M; T^*M)\cap\ker(I_1)$, then there exists $p\in C^3(M)$ such that  $p|_{\partial M}\equiv 0$ and  $f=dp$.

\item[(iii)] If $g$ is non-positively curved, $m\geq 2$, and $f\in C^\infty(M; \otimes^m_ST^*M)\cap\ker(I_m)$, then there exists $p\in C^\infty(M; \otimes_S^{m-1}T^*M)$ such that $p|_{\partial M}\equiv 0$ and $f=Dp$.
\end{itemize}
\end{thm} 

The condition of hyperbolicity for the trapped set $\mc{K}$ will be recalled in Section \ref{hyperK}. It is worthwhile to mention that the trapped set $\mc{K}$ is hyperbolic and there are no conjugate points whenever $(M,g)$ has negative sectional curvature. In the non-trapping situation, i.e. $\mc{K}=\varnothing$, we can even define $I_m$ on $L^1(M;\otimes_S^mT^*M)$ and Theorem~\ref{th4}(i) then holds with $f\in L^1(M)$  (see Remark~\ref{assreg}), while for tensors Theorem~\ref{th3} holds for $f\in H^1(M; \otimes_S^mT^*M)$  (see Theorem~\ref{injImsurf}). As in \cite[Theorem 1]{Guillarmou:2017if}, Theorem \ref{th4} easily implies deformation rigidity.

To put our theorems into perspective, let us recall some related results. 
Under the additional assumption that $\pl M$ is strictly convex, Theorem \ref{th2} was proved by Pestov-Uhlmann in \cite{Pestov:2005jo}, while Theorem \ref{th1} follows from \cite{Pestov:2005jo} together with a result of Mukhometov \cite{Muhometov:1981vs}. Within the class of negatively curved Riemannian metrics on compact 2-dimensional disks, the boundary rigidity was proved by Croke \cite{Croke:1990tk} and Otal \cite{Otal:1990ko}; such metrics are necessarily non-trapping.
On compact Riemannian surfaces with strictly convex boundary, no conjugate points, and hyperbolic trapped set, Theorem \ref{th2} was recently proved by the first author in \cite{Guillarmou:2017if}. In dimension  $n\geq 3$, there are several boundary rigidity results in the literature: here, we mention the one of Stefanov-Uhlmann-Vasy \cite{Stefanov:2017lt} for non-positively curved metrics, and the one of Stefanov-Uhlmann \cite{Stefanov:2005sy} for analytic metrics. Both these results require the boundary to be strictly convex, and removing this assumption is in general a serious issue. The determination of the $C^\infty$ jet of the Riemannian metric $g$ at $\pl M$ from the lens data was proved by Stefanov-Ulmann \cite{Stefanov:2009lp} for non-trapping metrics with no conjugate points. However, Zhou \cite{Zhou:2012mo} showed  that, without these assumptions, there are examples where the $C^\infty$ jet cannot be determined. Burago-Ivanov \cite{Burago:2010rp} proved a boundary rigidity result without any convexity assumption on the boundary, but assuming that the manifold is a compact domain in the Euclidean space and the  metric is close to a flat one. 
For $m=0,1$ and $f\in C^\infty(M; \otimes_S^mT^*M)$, Theorem~\ref{th3} was proved by Dairbekov \cite{Dairbekov:2006uf}. Under the extra assumption that $\pl M$ is strictly convex, then $M$ is a ball and Theorem~\ref{th3} was proved by Mukhometov \cite{Muhometov:1981vs} for $m=0$, Anikonov-Romanov \cite{Anikonov:1997mi} for $m=1$, and 
Paternain-Salo-Uhlmann \cite{Paternain:2013fx} for $m\geq 2$.   
 Theorem~\ref{th4} was proved by Pestov-Sharafutdinov \cite{Pestov:1988zq}, for non positively curved   balls with strictly convex boundary, and by Stefanov-Uhlmann-Vasy \cite{Uhlmann:2016qt, Stefanov:2014jw} when  $m\leq 2$ and $(M,g)$ admits a strictly convex foliation; all these cases are non-trapping. For the trapping case,
Theorem~\ref{th4} was proved by the first author in \cite[Theorem~3]{Guillarmou:2017if} under the further strict convexity assumption on $\pl M$.

The main difficulty in our work is the analysis of the glancing and the trapped trajectories where the exit time $\tau_g^+$ has discontinuities. In the non-trapping case, we use the approach of \cite{Pestov:2005jo} to prove the non-linear results (Theorems \ref{th1} and \ref{th2}). The main step is to prove the surjectivity of the adjoint $I_0^*$. This is in general harder than proving the injectivity of $I_0$ and, unlike in \cite{Pestov:2005jo}, in our case we cannot employ the operator  $I_0^*I_0$, since it is not pseudo-differential anymore when $\pl M$ is not strictly convex. 
In order to deal with this issue, we introduce a modified normal operator $P$ that replaces $I_m^*I_m$ for $m=0,1$ and has good Fredholm properties. By means of $P$, we are able to show the surjectivity of $I_0^*$ and $I_1^*$. The construction of $P$ requires 
 microlocalization away and near the glancing.
 In order to deal with the trapping case, we employ some results from \cite{Dyatlov:2016zh, Guillarmou:2017if} and the isolating blocks method of Conley-Easton \cite{Conley:1971eu}. Finally, we analyse the relations between the  boundary distance function and the lens data for simply connected manifolds with no conjugate points.

\subsection{Acknowledgements} We are grateful to C. Croke, G. Paternain and G. Uhlmann for useful discussions. We particularly thank X. Zhou for allowing us to include his argument for deriving Theorem~\ref{th2bis} from Theorem~\ref{th2}. C. G. is partially supported by the ANR project ANR-13-JS01-0006, and by the ERC consolidator grant IPFLOW. M. M. is partially supported by the ANR COSPIN (ANR-13-JS01-0008-01).

\section{Dynamical and geometric preliminaries}

\subsection{Exit time functions}

Let $\mc{M}$ be a smooth compact Riemannian manifold with non-empty boundary. Hereafter, we denote its interior by $\mc{M}^\circ$. We consider a nowhere vanishing smooth vector field $X$, with flow $\varphi_t$. Without loss of generality, we can assume that $\mc{M}$ is embedded inside an auxiliary closed manifold $\cN$ of the same dimension, and that $X$ is defined on the whole $\cN$, so that its flow $\varphi_t$ is complete. We consider the \emph{forward exit time function} $\tau_{\mc{M}}^+:\mc{M}\to [0,+\infty]$ and the \emph{backward exit time function} $\tau^-_{\mc{M}}:\mc{M}\to [-\infty,0]$, given by 
\begin{equation}\label{exittime}
\tau_{\mc{M}}^{\pm}(y):=
\pm \sup\big\{ t\geq 0\ \big|\  \varphi_{\pm s}(y)\in \mc{M}\ \  \forall s\in[0,t]\big\}.
\end{equation}
The forward trapped set $\Gamma_-(\mc{M})$ and the backward trapped set $\Gamma_+(\mc{M})$ of the flow are given by
\begin{align*}
\Gamma_-(\mc{M}):= \big\{y\in \mc{M}\ \big|\ \tau^+_{\mc{M}}(y)=+\infty\big\}, \quad 
\Gamma_+(\mc{M}):= \big\{y\in\mc{M}\ \big|\ \tau^-_{\mc{M}}(y)=-\infty\big\}. 
\end{align*}
Their intersection, is the trapped set
\[ \mc{K}(\mc{M}):=\Gamma_+(\mc{M})\cap \Gamma_-(\mc{M})=\bigcap_{t\in\rr}\varphi_t(\mc{M}).\] 
When the ambient manifold $\mc{M}$ is clear from the context, we will omit it from the notation and simply write $\Gamma_\pm$ and $\mc{K}$. The sets $\Gamma_\pm$ and $\mc{K}$ are closed in $\mc{M}$, and $\mc{K}$ is invariant by the flow $\varphi_t$. Moreover, if we denote by $d$ any Riemannian distance on $\mc{N}$, 
\begin{equation}\label{disttoK} 
d(\varphi_t(y),\mc{K})\to 0 \textrm{ as }t\to \mp \infty, \qquad\forall y\in \Gamma_\pm,
\end{equation}
see, e.g., \cite[Lemma 2.3]{Dyatlov:2016zh}. In what follows, we will always assume that
\begin{equation}\label{H1}
 \mc{K} \subset \mc{M}^\circ.
\end{equation}
We say that the flow $\varphi_t$ is \emph{non-trapping} on $\mc{M}$ when $\mc{K}=\varnothing$, which implies $\Gamma_-=\Gamma_+=\varnothing$ as well. Notice that $\tau_{\mc{M}}^{+}(y)$ and $-\tau_{\mc{M}}^{-}(y)$ are upper-semicontinuous, and therefore, since $\mc{M}$ is compact, the non-trapping condition is equivalent to the fact that $\tau_{\mc{M}}^{+}$, or equivalently $\tau_{\mc{M}}^{-}$, is uniformly bounded. 

Later on, we will need the following stability property of trapped sets.
\begin{lem}\label{l:trapped}
Let $\mc{M}_n\subset \mc{N}$ for $n\geq 0$ be a sequence of compact subsets such that 
$\mc{M}_{n+1}\subset \mc{M}_n$ and $\bigcap_{n\in\nn} \mc{M}_n = \mc{M}$.
If $\mc{K}(\mc{M})\subset \mc{M}^\circ$, then $\mc{K}(\mc{M}_n)=\mc{K}(\mc{M})$ for all $n\in\nn$ large enough.
\end{lem}

\begin{proof}
Assume by contradiction that, for infinitely many $n\in\nn$, there exists $y_n\in \mc{M}_n\setminus \mc{K}(\mc{M})$ such that $\tau_{\mc{M}_n}^+(y_n)=\infty$. Since $y_n\not\in \mc{K}(\mc{M})$, there exists $t_n\geq0$ such that $z_n:=\varphi_{t_n}(y_n)\not\in\mc{M}$. Notice that $\tau_{\mc{M}_n}^+(z_n)=\infty$. Up to extracting a subsequence, we have $z_n\to z\in \partial\mc{M}$. Since the exit time functions are upper semi-continuous, $\tau_{\mc{M}_n}^+(z)=\infty$ for all $n\in\nn$, and therefore $\tau_{\mc{M}}^+(z)=\infty$. This implies that $z\in \mc{K}(\mc{M})$, which contradicts the fact that $\mc{K}(\mc{M})\subset \mc{M}^\circ$.
\end{proof}

Let $\rho:\mc{M}\to[0,\infty)$ be a smooth boundary defining of $\pl \mc{M}$, i.e. $\partial\mc{M}=\rho^{-1}(0)$ and $d\rho|_{\pl \mc{M}}(y)\not=0$ for all $y\in \pl \mc{M}$. 
We partition $\partial \mc{M}$ into the incoming boundary $\partial_- \mc{M}$, the outgoing boundary $\partial_+ \mc{M}$, and the glancing boundary $\partial_0 \mc{M}$, given by 
\[ \pl_\mp \mc{M}=\{ y\in \pl\mc{M}\ |\ \pm X\rho(y)>0\}, \quad 
\pl_0\mc{M}=\{ y\in \pl\mc{M}\ |\ X\rho(y)=0\}.\]
We denote the positive and negative flowouts of $\pl_0\mc{M}$ by 
\begin{equation}\label{Gpm}
\mc{G}_\pm=\big\{\varphi_{\pm t}(y)\in \mc{M}\ \big|\  y\in\pl_0\mc{M},\  t\in[0,|\tau^{\pm}_{\mc{M}}(y)|]\big\},
\end{equation}  
and we set
\begin{align*}
 \mc{G}:=\mc{G}_+\cap \mc{G}_-.
\end{align*}
Both $\mc{G}_+$ and $\mc{G}_-$ are closed in $\mc{M}\setminus \Gamma_\pm$.
By the implicit function theorem, we have that
\begin{equation}\label{regulariteTau}
\tau^\pm_{\mc{M}}|_{\mc{M}\setminus (\Gamma_\mp\cup \mc{G}_\mp)}\in C^\infty( \mc{M}\setminus (\Gamma_\mp\cup \mc{G}_\mp)),
\end{equation}
whereas the restrictions $\pm \tau^\pm_{\mc{M}}|_{\mc{M}\setminus \Gamma_\mp}$ are only upper semi-continuous.

\subsection{Convex neighborhoods of the trapped set}

A result of Conley and Easton shows that the trapped sets in the interior of a compact manifold admit open neighborhoods satisfying a certain convexity condition.

\begin{prop}[Theorem 1.5 in \cite{Conley:1971eu}]\label{t:conley}
If $\mc{K}\subset\mc{M}^\circ$, then $\mc{K}$ admits a compact neighborhood $\mc{U}\subset \mc{M}^\circ$ with smooth boundary $\partial \mc{U}$ such that 
$\partial_0\mc{U}:=\big\{y\in\partial \mc{U}\ \big|\ X(y)\in T_y\partial \mc{U}\big\}$
is a smooth submanifold of $\partial\mc{U}$ of codimension 1. Moreover, for every $y\in\partial_0\mc{U}$ there exists $\epsilon=\epsilon(y)>0$ such that $\varphi_t(y)\not\in\mc{U}$ for all $t\in[-\epsilon,\epsilon]\setminus\{0\}$.
\hfill\qed
\end{prop}

Notice, however, that $\mc{U}$ itself is not convex for the dynamics: there may be points $y,\varphi_t(y)\in \mc{U}$ such that $\varphi_s(y)\not\in \mc{U}$ for some $s\in(0,t)$. 
In order to use the results of \cite{Dyatlov:2016zh} on the resolvent of the flow, we need a stronger (quadratic) convexity condition for the boundary of $\mc{U}$. We shall achieve this condition by slightly  modifying the vector field $X$ near $\pl_0\mc{U}$; the changes in the dynamics will be negligible for our purposes.

\begin{lem}\label{modifconley}
There exists a vector field $X_0$ on $\mc{M}$ such that $X-X_0$ is arbitrarily small in the $C^\infty$ topology and supported in an arbitrarily small neighborhood $\mc{O}$ of $\pl_0\mc{U}$, with the following property: if $r\in C^\infty(\mc{U};[0,\infty))$ is a boundary defining function of $\pl \mc{U}$,  then for each $y\in \pl\mc{U}$ such that $X_0r(y)=0$ we have $X_0^2r(y)<0$.   Moreover, the backward and forward trapped sets for the flow of $X_0$ on $\mc{U}$ are those of $X$, i.e. $\Gamma_\pm(\mc{U})$.
\end{lem}
\begin{proof} 
We consider a smooth function $s\in C^\infty(\mc{U};\rr)$ without critical points on $\pl_0\mc{U}$, satisfying $\pl_0\mc{U}=\pl_0\mc{U}\cap s^{-1}(0)$ and 
\begin{align*}
\pl_\pm \mc{U}
:=
\big\{ y\in \pl \mc{U}\ \big|\  \mp Xr(y)>0\big\}
=
\big\{y\in \pl \mc{U}\ \big|\ \mp s(y)>0\big\}.
\end{align*}
Let $\mc{O}\subset \mc{M}$ be an arbitrarily small open neighborhood of $\pl_0\mc{U}$,  and 
$\chi\in C_c^\infty(\mc{O};\rr^+)$ a smooth function identically equal to $1$ on a neighborhood $\pl_0\mc{U}$. We require $\mc{O}$ to be small enough so that $r|_{\mc{O}}$ has no critical points.  We introduce an arbitrary Riemannian metric $G$ on $\mc{M}$ and define the vector field $R:=\|\nabla r\|^{-2}\nabla r$, where $\nabla$ denotes the gradient and $\|\cdot\|$ the norm with respect to $G$.  Let $S$ be a smooth vector field on $\mc{M}$, tangent to $\pl \mc{U}$, that satisfies $ds(S)>0$ and $G(S,R)\equiv0$ in a neighborhood of $\pl_0\mc{U}$. We set
\begin{align*}
X_0:=X+\eps\chi (sR-S).
\end{align*}
Since $\mp Xr>0$ in $\partial_\pm\mc{U}$ and $(sR-S)r=s$, we infer that for $y\in \pl \mc{U}$, $X_0r(y)=0$ if and only if $y\in\partial_0\mc{U}$. For each $y\in \pl_0\mc{U}$, we have
\begin{align*}
X_0^2r(y)
& = 
X^2r(y) + \eps^2(sR-S)^2 r(y) + \eps (sR-S)Xr(y) + \eps X (sR-S)r(y)\\
& = 
X^2r(y) - \eps^2 Ss(y) - \eps SXr(y) + \eps Xs(y).
\end{align*}
Proposition \ref{t:conley} implies that $X^2r(y)\leq0$. Moreover, $Xs(y)\leq 0$ (otherwise, by Proposition \ref{t:conley} and a continuity argument, there would be a point $y'\in \mc{U}$ close to $y$ and a small $t>0$ such that $\varphi_t(y')\in\pl_-\mc{U}$). Since $\pm Xr>0$ and $\pm s>0$ in $\partial_\mp\mc{U}$, and since $Ss>0$, we also have that $SXr(y)\geq0$. Summing up, we have 
\begin{align*}
 X_0^2r(y) \leq \eps^2 Ss(y)<0.
\end{align*}
By Proposition \ref{t:conley} and the continuity of the flow $\varphi_t$, every $y\in \pl_0\mc{U}$ has a  neighborhood $B_y$ such that, for each $y'\in B_y$, the orbit of $y$ exits $\mc{U}$ for some (small) positive and negative times. Since $\partial_0\mc{U}$ is compact, we can choose the above neighborhood $\mc{O}$ to be contained in a finite union of such $B_y$'s. In particular, $\Gamma_\pm(\mc{U})\cap\mc{O}=\varnothing$. If we choose $\epsilon>0$ sufficiently small, the orbit of every $y'\in\mc{O}$ with respect to the flow of $X_0$ still exits $\mc{U}$ in both positive and negative time. This proves that the backward and forward trapped sets for the dynamic of $X_0$ in $\mc{U}$ are still $\Gamma_\pm(\mc{U})$.
\end{proof}

We denote by $\phi_t$ the flow of $X_0$ defined in the Lemma~\ref{modifconley} and let 
$\tau_{0,\mc{U}}^\pm$ and $\tau_{0,\mc{M}}^\pm$ be respectively the exit times from $\mc{U}$ and $\mc{M}$ for the modified flow $\phi_t$.
\begin{lem}
\label{nbhd_of_K}
There exists an arbitrarily small neighbourhood $\mc{C}\subset \mc{U}^\circ$ of $\mc{K}$ such that, for each $y\in \mc{C}$, 
we have $\phi_t(y) \notin \mc{C}$ for each $t\in(\tau^-_{0,\mc{M}}(y),\tau^-_{0,\mc{U}}(y))\cup(\tau^+_{0,\mc{U}}(y),\tau^+_{0,\mc{M}}(y))$.
\end{lem}
\begin{proof}
Let us assume, by contradiction, that such a neighborhood $\mc{C}$ does not exist. 
Then, there exist sequences  $y_n\in \mc{U}$ and  $t_n\in(\tau^-_{0,\mc{M}}(y_n),\tau^-_{0,\mc{U}}(y_n))\cup(\tau^+_{0,\mc{U}}(y_n),\tau^+_{0,\mc{M}}(y_n))$
 such that $y_n\to y\in\mc{K}$ and $z_n:=\phi_{t_n}(y_n)\to z\in\mc{K}$. Up to extracting a subsequence, we can assume that all the $t_n$'s have the same sign. Let us consider the case where $t_n>0$, the other case being entirely analogous. Since $\mc{K}\subset\mc{U}$, the continuity of the flow $\phi_t$ implies that 
 $\tau^\pm_{0,\mc{U}}(y_n)\to\pm\infty$ and $\tau^\pm_{0,\mc{U}}(z_n)\to\pm\infty$. Therefore, we have a sequence of points $w_n:=\phi_{\tau^+_{0,\mc{U}}(y_n)}(y_n)\in\partial\mc{U}$ such that
\[
\tau_{0,\mc{M}}^-(w_n)\leq \tau_{0,\mc{M}}^-(y_n)\to-\infty,\qquad
\tau_{0,\mc{M}}^+(w_n)\geq \tau_{0,\mc{M}}^+(z_n)\to\infty.
\]
Since $\pl\mc{U}$ is compact, up to extracting a subsequence we have that $w_n\to w\in \partial \mc{U}$. By the upper semi-continuity of the functions $\pm\tau_{0,\mc{M}}^\pm$, we infer that 
\[\pm\tau_{0,\mc{M}}^\pm(w)\geq \lim_{n\to\infty}\pm\tau_{0,\mc{M}}^\pm(w_n)=
+\infty.\] 
This shows that $w\in\mc{K}$, which contradicts the fact that $\partial \mc{U}\cap\mc{K}=\varnothing$.
\end{proof}

\subsection{Hyperbolic trapped set}\label{hyperK}

We recall that the trapped set $\mc{K}$ is said to be \emph{hyperbolic} when there exist $C>0$, $\nu>0$, and a continuous splitting
\begin{align*}
T_y\mc{M}= \mathrm{span}\{X(y)\}\oplus E_u(y) \oplus E_s(y),\qquad
\forall y\in \mc{K},
\end{align*}
that is invariant under the linearized flow $ d\varphi_t$, and satisfies 
\begin{align*}
\|d\varphi_t(y)w\| \leq C\,e^{-\nu t}\|w\|,
\qquad
\forall w\in E_s(y),\ t\geq 0,\\
\|d\varphi_{-t}(y)w\| \leq C\,e^{-\nu t}\|w\|,
\qquad
\forall w\in E_u(y),\ t\geq 0.
\end{align*}
In the above inequalities, the norm is the one associated to any fixed Riemannian metric $G$ on $\mc{M}$. By~\cite[Lemma~2.10]{Dyatlov:2016zh}, the continuous subbundle $E_s$ admits a continuous extension $E_-$ over $\Gamma_-$ that is invariant under the linearized geodesic flow and, up to lowering the constants $C>0$ and $\nu>0$, satisfies
\begin{align*}
\|d\varphi_t(y)w\| \leq C\,e^{-\nu t}\|w\|,
\qquad
\forall w\in E_-(y),\ t\geq0.
\end{align*}
The continuous subbundle $E_u$ admits a continuous extension $E_+$ over $\Gamma_+$ with similar properties in backward time. We also define the dual spaces $E_\pm^*\subset T^*\mc{M}$ by 
\begin{equation}\label{dualspa}
E_\pm^*(\rr X\oplus E_\pm)=0.
\end{equation}
If $\mc{K}\subset \mc{M}^\circ$ is a hyperbolic set, and both vector field $X$ and $X_0$ have zero divergence near $\Gamma_\pm(\mc{U})$ with respect to a smooth measure $\mu$,  we can use  \cite[Proposition~2.4]{Guillarmou:2017if} to obtain that  $\mu(\Gamma_\pm(\mc{U}))=0$. We claim that
\begin{equation}\label{muGamma}
\mu(\Gamma_\pm)=0.
\end{equation}
Indeed, if $z\in \Gamma_-$, the convergence \eqref{disttoK} implies that there is a small ball $B_z$ centered in $z$ and $T>0$ such that $\varphi_T(B_z)\subset \mc{U}$; therefore 
\[\mu(\Gamma_-\cap B_z)=\mu(\varphi_T(\Gamma_-\cap B_z))\leq\mu( \Gamma_-(\mc{U}))=0.\]

\subsection{Geodesic flows}
The setting that we have introduced so far in this section was rather general. We are now going to focus to the case of geodesic flows, which is the one we need for the applications of this paper. Let $(M,g)$ be an $n$-dimensional compact Riemannian manifold with boundary, with $n\geq 2$. We will always assume $M$ is oriented (many of the obtained results later hold without orientation by taking a finite cover). Without loss of generality, we can consider $M$ as a compact subset of a closed Riemannian manifold $(N,g)$ of dimension $n$. We denote the associated unit tangent bundles by $\mc{M}:=SM$ and $\mc{N}:=SN$. The vector field $X$ and its flow $\varphi_t$ will now be the geodesic vector field and the geodesic flow of $\mc{N}$. As before, we assume that the trapped set $\mc{K}=\mc{K}(\mc{M})$ is contained in the interior $\mc{M}^\circ$. In the introduction, we defined the non-trapping property by the condition $\tau^+_g(y)<\infty$ for all $y\in \pl \mc{M}$ with $\tau_g^+:=\tau^+_{\mc{M}}|_{\pl \mc{M}}$. This is actually equivalent to
 the condition $\tau^+_{\mc{M}}(y)<\infty$ for all $y\in \mc{M}$. Indeed one sense is obvious, while if $\tau_g^+$ is bounded then for $(x,v)\in SM\setminus (\Gamma_-\cup B)$ with $B:=\{\varphi_t(y)\in SM\, \big|\ y\in \pl_-SM\cup \pl_0SM, t\in [0,\tau^+_g(y)]\}$, we have $(x,-v)\in B\cap \Gamma_-$, which is not possible since $\Gamma_-$ and $B$ are two disjoint closed sets satisfying $SM=B\cup \Gamma_-$. 

With a common abuse of notation, we will often denote the points of $\mc{N}$ as pairs $(x,v)$, where $x\in N$ and $v\in S_xN$. We denote by $\pi_0:SM\to M$ the base projection $\pi_0(x,v)=x$.
The tangent space of $\mc{N}$ splits as the direct sum of three vector subbundles 
\begin{equation}\label{splitting} 
T\mc{N}=\mathrm{span}\{ X\}\oplus \mc{V}\oplus \mc{H},
\end{equation}
where $\mc{V}:=\ker d\pi_0$ is the vertical bundle, whereas $\mc{H}$ is the horizontal bundle obtained as the kernel of the Levi-Civita connection map $\kappa:\ker \alpha \to \mc{N}$. If $\mc{Z}\to \mc{N}$ is the bundle with fibers $\mc{Z}_{y}=\{ v\in T_{\pi_0(y)}N\ |\  g(\pi_0(y),v)=0\}$, the maps
$d\pi_0: \mc{H}\to \mc{Z}$ and $\kappa:\mc{V}\to \mc{Z}$ are isomorphisms. 
The Sasaki metric $G$ on $S^*N$ is defined by 
\begin{equation*}
G(\zeta,\zeta')=g(d\pi_0\zeta,d\pi_0\zeta')+g(\kappa\zeta,\kappa\zeta').
\end{equation*}
The unit tangent bundle $\mc{N}$ admits a contact 1-form $\alpha$, called the Liouville form and defined by $\alpha_{(x,v)}(w)=g_x(v,d\pi_0(x,v)w)$. The Riemannian volume form of $(\mc{N},G)$ is given by 
\begin{align}
\label{e:contact_volume_form}
\tfrac{1}{(n-1)!}\alpha\wedge (d\alpha)^{n-1}.
\end{align}
The associated measure $\mu$ on $\mc{N}$ is called the Liouville measure. The geodesic vector field $X$ is the Reeb vector field associated to $\alpha$, i.e $\alpha(X)=1$ and $i_Xd\alpha=0$. In particular, $X$ has zero divergence with respect to the volume form~\eqref{e:contact_volume_form}, and therefore the geodesic flow $\varphi_t$ preserves the measure $\mu$. We refer the reader to, e.g., \cite[Chapter~1]{Paternain:1999fc} for the background about geodesic flows.

The flowouts $\mc{G}_\pm$ of $\pl_0\mc{M} = S(\pl M)$ are closed subsets of zero measure in $\mc{M}\setminus\Gamma_\pm$, i.e.\
$\mu(\mc{G}_\pm)=0$.
Indeed, $\pl_0\mc{M}$ is a smooth submanifold of $\mc{N}$ of codimension $2$, and the flowouts $\mc{G}_\pm$ are contained in the image of the map
\begin{equation}\label{map1}
(-\infty,\infty)\times\pl_0\mc{M}\to \mc{N},\qquad
(t,y)\mapsto \varphi_t(y),
\end{equation}
which has zero measure according to Sard's Theorem.

Two (not necessarily distinct) points $y,y'\in \mc{N}$ are \emph{conjugate} when there exists $t>0$ such that $\varphi_t(y)=(y')$ and $d\varphi_t(y)\mc{V}_{y}\cap \mc{V}_{y'}\neq\{0\}$. With a common abuse of terminology, we also say that $\pi_0(y)$ and $\pi_0(y')$ are conjugate points. We say that $(M,g)$ has no conjugate points when no pair of points in it are conjugate along a geodesic segment contained in $M$. The following lemma is  proved in \cite[Lemma~2.3]{Guillarmou:2017if} (the convexity assumption on $\pl M$ was superfluous in the proof).

\begin{lem}\label{l:no_conjugate_points}
Assume that $(M,g)$ does not contain conjugate points, and the trapped set $\mc{K}$ is hyperbolic and contained in the interior $\mc{M}^\circ$. Then, for every sufficiently small neighborhood $M_1 \subset N$ of $M$, $(M_1,g)$ does not contain conjugate points neither.
\hfill\qed
\end{lem}

\section{The X-ray transform}

Throughout this section, $(M,g)$ is a smooth manifold with boundary, and we assume that the trapped set $\mc{K}$ is a (possibly empty) compact hyperbolic set contained in the interior $\mc{M}^\circ= SM^\circ$.

\subsection{Resolvents and X-ray transform}
We start by defining the two natural inverses for the first order differential operator $-X$ in $\mc{M}$:
\[\begin{gathered} 
R_{\mc{M}}^\pm: C_c^0(\mc{M}\setminus \Gamma_\pm)\to L^\infty(\mc{M}), \qquad 
R_{\mc{M}}^\pm f(y):=\int_0^{\tau_{\mc{M}}^\pm(y)}f(\varphi_t(y))dt.
\end{gathered}\]
We call them the forward resolvent $R_{\mc{M}}^+$ and backward resolvent $R_{\mc{M}}^-$. For each $f\in C_c^0(\mc{M}\setminus \Gamma_\pm)$, we have 
\begin{equation}\label{eqresolvent}
\left\{\begin{array}{l}
-XR_{\mc{M}}^\pm f=f \,\, \textrm{ in }\mc{D}'(\mc{M}^\circ)\\ 
(R_{\mc{M}}^\pm f)|_{\pl_\pm \mc{M}}=0.
\end{array}\right.
\end{equation}
The operators $R_{\mc{M}}^\pm$ also make sense as maps  
\[R_{\mc{M}}^\pm: C^0(\mc{M})\to L^\infty_{\rm loc}(\mc{M}\setminus \Gamma_\mp).\]
\begin{defi}
The X-ray transform acting on functions on the unit tangent bundle $\mc{M}=SM$ is defined as the bounded operator 
\begin{equation}\label{defXray}
\begin{gathered} 
I : C^0(\mc{M})\to L^\infty_{\rm loc}((\pl_-\mc{M}\cup \pl_0\mc{M})\setminus \Gamma_-), \\ 
If(y):=\int_0^{\tau^+_{\mc{M}}(y)}f(\varphi_t(y))dt=R_{\mc{M}}^+f(y)
\end{gathered}\end{equation}
\end{defi}

\begin{lem}\label{outsideGamma}
For each $k\in\nn \cup\{0\}$, the resolvents $R_{\mc{M}}^\pm$ extend as continuous maps
\begin{equation}
\label{regulRm}
\begin{split}
R_{\mc{M}}^\pm &: C^k(\mc{M})\to C^k(\mc{M}\setminus (\mc{G}_\mp\cup \Gamma_\mp)),\\
R_{\mc{M}}^\pm &: C_c^k(\mc{M}\setminus \Gamma_\pm)\to C^k(\mc{M}\setminus \mc{G}_\mp),\\
R_{\mc{M}}^\pm &: H^k(\mc{M})\to H^k_{\rm loc}(\mc{M}\setminus (\mc{G}_\mp\cup \Gamma_\mp)),\\
R_{\mc{M}}^\pm &: H_{\rm comp}^k(\mc{M}\setminus \Gamma_\pm)\to H^k_{\rm loc}(\mc{M}\setminus \mc{G}_\mp),
\end{split}
\end{equation} 
and, for each  $p\in[1,\infty)$, as continuous maps 
\[R_{\mc{M}}^\pm: C^0(\mc{M})\to L^p(\mc{M}).\]
\end{lem}
\begin{proof} The continuity of $R_{\mc{M}}^\pm: C^k(\mc{M})\to C^k(\mc{M}\setminus (\mc{G}_\mp\cup \Gamma_\mp))$ 
 follows directly from the regularity \eqref{regulariteTau} of the map $\tau_{\mc{M}}^\pm$.  
The continuity  $R_{\mc{M}}^\pm: C_c^k(\mc{M}\setminus \Gamma_\pm)\to C^k(\mc{M}\setminus \mc{G}_\mp)$ follows similarly, using in addition that there exists $T>0$ depending only on $d(\supp(f), \Gamma_\pm)$ such that
$f(\varphi_{\pm t}(y))=0$ for all $y\in \mc{M}$ and $|\tau_{\mc{M}}^{\pm}(y)|\geq t\geq T$.
The continuity in $H^k$ is proved analogously.
For the last statement, it suffices to prove that $\tau_{\mc{M}}^\pm\in L^p(\mc{M})$ for all $p<\infty$. For large $T>0$, consider the set $V_{\mc{M}}(T)=\{y\in \mc{M}\ |\ \tau^+_{\mc{M}}(y)>T\}$. By the upper continuity of $\tau^+_{\mc{M}}$, for $T>0$ large $V_{\mc{M}}(T)$ is contained in a small neighborhood of $\Gamma_-$ and by 
\cite[Section 2.4]{Guillarmou:2017if}, we have $\mu(V_{\mc{M}}(T))=\mc{O}(e^{-\nu T})$ for some $\nu>0$, and thus $\tau^+_{\mc{M}}$ is in $L^p$. The same holds with $\tau_{\mc{M}}^-$.
\end{proof}

We denote by $n$ and by $\nu$ the respective inward-pointing unit normal vector 
fields to $\pl \mc{M}$ and $\pl M$ with respect to the Sasaki metric $G$ on $\mc{M}=SM$ and $g$ on $M$. As before, we denote by $\mu$ Liouville measure, 
and let $\mu_H$ be the Riemannian measure on $(\partial\mc{M},H=G|_{\pl \mc{M}})$. We also consider the measure $\mu_\nu$ on $\partial \mc{M}$ given by
\begin{align}
\label{e:measure_at_boundary}
d\mu_{\nu}(x,v):=|g(\nu,v)|\,d\mu_{H}(x,v).
\end{align}
\begin{lem}[Santalo's formula]\label{santalofor}
The X-ray transform $I$ extends as a bounded map $I:L^1(\mc{M},d\mu)\to L^1(\pl_-\mc{M}\cup \pl_0\mc{M},d\mu_\nu)$, and for all $f\in L^1(\mc{M},d\mu)$ we have
\begin{equation}\label{santalo}
\int_{\mc{M}}fd\mu= \int_{\pl_- \mc{M}}If\, d\mu_\nu.
\end{equation}
\end{lem}
\begin{proof}
Consider the open set $\mc{T}:=\mc{M}\setminus (\mc{G}_-\cup \Gamma_+\cup \Gamma_-)$, which has full $\mu$-measure in $\mc{M}$.
Each $f\in C_c^\infty(\mc{T})$ can be written as $f=-Xu$ in $\mc{M}=SM$, where $u:=R_{\mc{M}}^+f$. Notice that  $u$ extends as a smooth function in $\mc{M}$ satisfying $u|_{\pl_+\mc{M}}=0$ and $u|_{\pl_-\mc{M}\cup \pl_0\mc{M}}=If$. By Green's formula, we have
\begin{align}\label{e:Green_formula}
\int_{\mc{M}}f\,d\mu
= 
-\int_{\mc{M}}Xu\, d\mu
=
\int_{\pl_- \mc{M}}
u\, G(X,n)\,d\mu_{H}.
\end{align}
Since $\mc{V}=\ker d\pi_0$ is tangent to $\pl \mc{M}$, we have $G(n,\mc{V})=0$. Therefore
\begin{align*}
G(X(x,v),n(x,v))
=
g(d\pi_0(x,v)X(x,v),d\pi_0(x,v)n(x,v))
=
g(v,\nu(x)),
\end{align*}
which, together with~\eqref{e:Green_formula}, implies that~\eqref{santalo} holds for all $f\in C_c^\infty(\mc{T})$. Since $C_c^\infty(\mc{T})$ is dense in $L^1(\mc{M},d\mu)$, and since $|If|\leq I(|f|)$ pointwise, $I$ extends as a bounded operator 
$I:L^1(\mc{M},d\mu)\to L^1(\pl_-\mc{M}\cup \pl_0\mc{M},d\mu_\nu)$, and \eqref{santalo} holds for all  $f\in L^1(\mc{M},d\mu)$.
\end{proof}

\subsection{Regularity properties of the resolvent.}

The main result of this section is the following Livsic-type statement.
\begin{prop}\label{Livsic}
For $k\in \nn$, let $f\in H^k(\mc{M})\cap C^0(\mc{M})$ be a function satisfying $\pl_n^{j}f|_{\pl \mc{M}}=0$ for all $j=0,\dots,k-1$, where $\pl_n$ is the inner normal derivative at $\pl \mc{M}$. If $If=0$, then there exists a unique $u\in H^k(\mc{M})$ such that $u|_{\pl\mc{M}}=0$
and $-Xu=f$. Moreover $u=R^+_{\mc{M}}f=R^-_{\mc{M}}f$. If $f\in C^k(\mc{M})$, then 
$u\in C^{k}(\mc{M}\setminus (\mc{K}\cup \mc{G}))$ as well.
\end{prop}
\begin{proof}
First, we take $M_1$ as in Lemma \ref{l:no_conjugate_points} and we can assume that $d_g(\pl M_1,M)=2\eps>0$ for some $\eps>0$ small; we will denote $\mc{M}_1:=SM_1$.
Using the assumptions on $f$, we can extend $f$ by $0$ in $\mc{N}$ in a way that  the extension $\til{f}$ remains in $H^k(\mc{N})$.
Since $If=0$, we have that 
$R^{+}_{\mc{M}}f=R^{-}_{\mc{M}}f$ in $\mc{M}\setminus (\Gamma_+\cup \Gamma_-)$: indeed for each point $y\in \mc{M}\setminus (\Gamma_+\cup \Gamma_-)$, we have 
$R^{+}_{\mc{M}}f(y)-R^-_{\mc{M}}f(y)=If(\varphi_{\tau^{-}_{\mc{M}}(y)}(y))=0$. 
Moreover $R^+_{\mc{M}}f|_{\pl_-\mc{M}\cup \pl_0\mc{M}}=If=0$, thus 
$u:=R^+_{\mc{M}}f$ vanishes on the whole boundary $\pl \mc{M}$. We also have $-Xu=f$ by the choice of $u$.

Next, we can use Lemma \ref{outsideGamma} to deduce that $u$ is $H^{k}_{\rm loc}\cap C^0$ on $\mc{M}\setminus (\mc{K}\cup \mc{G})$ where $\mc{G}=\mc{G}_+\cap \mc{G}_-$. 
Also, $u$ is  clearly the unique such function solving $-Xu=f$ with $u|_{\pl \mc{M}}=0$.
Notice that $\mc{G}\cap \mc{K}=\varnothing$ by our assumption that $\mc{K}\cap \pl\mc{M}=\varnothing$. Let us show that in fact $u\in H^{k}(\mc{M}\setminus \mc{K})$. For a point $y\in \mc{G}$, 
there is $\eps>0$ small such that $y_+:=\varphi_{\tau_{\mc{M}_1}^+(y)-\eps}(y)\in 
\mc{N}\setminus \mc{M}$ and there is a small hypersurfaces $S_+\subset \mc{N}\setminus \mc{M}$ passing through $y_+$ such that $X$ is transverse to $S_+$ in a small neighborhood $O_+$ of 
$y_+$ in $S_+$. For $y'$ near $y$, denote by $T_+(y')>0$ the smallest time so that 
$\varphi_{T_+(y')}(y')\in S_+$, clearly $y'\mapsto T_{+}(y')$ is smooth in a small neighborhood of $y$ by the implicit function theorem. 
The neighborhood $O_+$ is contained in $\mc{N}\setminus \mc{M}$ and  we claim that for $y'$ near $y$
\begin{equation}\label{simplified} 
R_+^{\mc{M}}f(y')=\int_{0}^{\tau_{\mc{M}}^+(y')}f(\varphi_t(y'))dt=\int_0^{T_+(y')}\til{f}(\varphi_t(y'))dt.
\end{equation}
Indeed, the integral curve 
\[\gamma(y'):=\bigcup_{\tau^{\mc{M}}_+(y')\leq t\leq T_+(y'))}\varphi_{ t}(y')\cap \mc{M}\]
is either empty and then the statement is obvious, or it is a disjoint union of geodesic segments
$(\gamma_j)_{j\in J}$ in $\mc{M}$ with endpoints on $\pl \mc{M}$, 
in which case for each $j$ we get $\int_{\gamma_j}f=0$ since $If=0$. This shows \eqref{simplified} and we conclude from the expression \eqref{simplified} of $u$ and the fact that $\til{f}\in H^k(\mc{N})$ 
that $u$ is in $H^k$ near $y$. Therefore $u\in H^{k}(\mc{M}\setminus \mc{K})$. The same argument also shows that if $f\in C^k$, then $u\in C^{k}(\mc{M}\setminus (\mc{K}\cup \mc{G}))$.

The last part consists in analyzing the regularity near $\mc{K}$, using propagation of singularities and the results of \cite{Dyatlov:2016zh, Guillarmou:2017if}. First, we split $f=f_1+f_2$ with $\supp(f_1)\cap \Gamma_+=\varnothing$ and $f_2$ supported in a small neighborhood of $\Gamma_+$,  $f_i$ having the same regularity as $f$. 
From the definition of $u_1:=R_{\mc{M}}^+f_1$ and the flow-invariance of $\Gamma_+$, we have $\supp(u_1)\cap \Gamma_+=\varnothing$. Therefore to study the regularity of $u=R_{\mc{M}}^+f$ near $\mc{K}$, its suffices to study the regularity of $u_2:=R_{\mc{M}}^+f_2$ near $\mc{K}$. We will show that, near $\mc{K}$, $u_2$ is microlocally $H^k$ everywhere except at $E_-^*$ (defined in \eqref{dualspa}). Here and below we use the following usual terminology: a function $f$ is microlocally $H^s$ in a conic open set $W$ of $T^*\mc{M}$ if for each pseudo-differential operator $A\in \Psi^0(\mc{M})$ with microsupport contained in $W$, $Af\in H^s(\mc{M})$.

Let $\mc{U}\subset \mc{M}$ be the subset of Proposition \ref{t:conley} and 
$\mc{C}\subset \mc{U}$ be the neighborhood of $\mc{K}$ obtained from Lemma \ref{nbhd_of_K}.
We also consider a neighborhood $\mc{C}'\subset \mc{C}$ of $\mc{K}$ where  
 the vector field $X_0$ in Lemma~\ref{modifconley} is equal to $X$ and the flow lines of 
 $X$ in $\mc{U}$ intersecting $\mc{C}'$ are contained in the set where $X_0=X$. 
 This is possible since there is a neighborhood of $\Gamma_\pm(\mc{U})$ where $X_0=X$ by Lemma~\ref{modifconley}, we shall just write $\varphi_t$ instead of $\phi_t$ for the flow of $X_0$ in those regions.
We can then assume that $f_2$ is supported in the set $\{y\in \mc{M}\ |\ \exists t\geq 0, \varphi_{-t}(y)\in \mc{C}', \varphi_{[-t,0]}(y)\in \mc{M}\}$ by choosing the support of $f_2$ close enough to $\Gamma_+$. We
then split $f_2=f_3+f_4$ where $\supp(f_3)\subset \mc{C}'$ and
$\varphi_{-T}(\supp(f_4))\subset \mc{C}'$ for some large $T>0$. Define 
$u_j:=R_{\mc{M}}^+f_j$ for $j=3,4$.  
Due to the property of $\mc{C}$ in Lemma \ref{nbhd_of_K} and the support of $f_3$, we have for $y\in \mc{U}\setminus \Gamma_-$
\[ u_3(y)=\int_{0}^{\tau_{\mc{U}}^+(y)}f_3(\varphi_t(y))dt\] 
where $\tau_{\mc{U}}^+(y)$ is defined as  $\tau_{\mc{M}}^+(y)$ but with $\mc{U}$ replacing $\mc{M}$. Notice then that $u_3=0$ in a neighborhood of $\pl_+\mc{U}$.
The first-order operator $X_0$ of Corollary~\ref{modifconley} satisfies the assumptions of the paper \cite{Dyatlov:2016zh} in the set $\mc{U}$. Therefore, by \cite{Dyatlov:2016zh}, 
it has two resolvents $R^\pm_{\mc{U}}(\la):C_c^\infty(\mc{U})\to \mc{D}'(\mc{U})$ solving 
$(-X_0\pm \la)R_{\mc{U}}^\pm (\la)={\rm Id}$ that are meromorphic 
in $\la \in \cc$, analytic in ${\rm Re}(\la)>C$ for some $C\geq 0$ and 
given in $\mc{U}$ by the converging expression
\[R_{\mc{U}}^\pm (\la)f_3(y)=\int_{0}^{\tau_{0,\mc{U}}^\pm(y)}e^{\mp \la t}f(\phi_t(y))dt.\] 
Moreover, as in Section 4.2 of \cite{Guillarmou:2017if}, the operators $R_{\mc{U}}^\pm (\la)$ are actually analytic in ${\rm Re}(\la)\geq 0$ and due to the support condition of $f_3$, for $y\in \mc{U}\setminus \Gamma_{-}$ 
\[R_{\mc{U}}^{+}(0)f_3(y)=\int_{0}^{\tau_{\mc{U}}^+(y)}f_3(\varphi_t(y))dt=u_3(y).\] 
By \cite[Proposition 4.3]{Guillarmou:2017if} (or \cite{Dyatlov:2016zh}, Proposition 6.1) applied to the vector field  $X_0$, we deduce that there is a neighborhood $Q_+$ of $E_+^*$ in $T^*\mc{U}$, conic 
in the fibers of $T^*\mc{U}$, such that $u_3$  is microlocally in $H^k(\mc{U})$ in $Q_+$. We also know by ellipticity that $u_3$ is microlocally $H^k$ outside the characteristic set $\{\xi \in T^*\mc{U}\ |\  \xi(X_0)=0\}$ of $X_0$. We next use propagation of singularities to prove that actually $u_3$ is microlocally $H^k$ outside $E_-^*$: by \cite[Lemma 2.10]{Dyatlov:2016zh}, for each $(y,\xi)\notin E_-^*$ with $y\in \mc{U}, \xi(X_0)=0$, either $\phi_{t}(y)\in \pl_-\mc{U}$ for some $t\geq 0$ or $\Phi_{t}(y,\xi):=(d\phi_{t}(y)^{-1})^T.\xi$ belongs to a small neighborhood of $E_+^*$ for some $t\geq 0$, thus since $u_3$ is microlocally $H^k$ in those regions of $T^*\mc{U}$, 
the propagation of singularities (see e.g. \cite[Lemma 3.2]{Dyatlov:2016zh}) gives us that $u_3$ is microlocally $H^k$ outside $E_-^*$ over $\mc{U}$.
Let us next consider the regularity of $u_4$. First, it is direct to check that in $\mc{M}$
\[  \varphi_{-T}^*R_{\mc{M}}^+(\varphi_T^*f_4)=R_{\mc{M}}^+f_4\]
and therefore, since ${\rm supp}(\varphi_T^*f_4)\subset \mc{C}'$, we can write 
\[ u_4=\varphi_{-T}^*{\bf }R_{\mc{U}}^+(0)(\varphi_T^*f_4).\]
Now, as for $u_3$, we have that ${\bf }R_{\mc{U}}^+(0)(\varphi_T^*f_4)$ is microlocally $H^k$ outside $E_-^*$ in $\mc{U}$, and by Egorov theorem (or again \cite[Lemma 3.2]{Dyatlov:2016zh}) we obtain that $u_3$ is microlocally $H^k$ outside $E_-^*$ over $\mc{M}$. We have thus proved that $u_2=u_3+u_4$ is microlocally $H^k$ outside $E_-^*$ over $\mc{U}$, and thus the same holds for $u$. 

To conclude the argument, we repeat the argument by splitting $f=f_1'+f'_2$ as above, except that $\supp(f_1')\cap \Gamma_-=\varnothing$ and $f_2'$ is supported in a small neighborhood of $\Gamma_-$, and we use the resolvent $R_{\mc{U}}^-(0)$ instead of 
$R_{\mc{U}}^+(0)$: since, by \cite[Proposition 4.3]{Guillarmou:2017if}, for each $h\in H_{\rm comp}^k(\mc{U})$, 
$R_{\mc{U}}^-(0)h$ is microlocally in $H^k$ in a conic neighbohood $Q_-$ of $E_-^*$, the same argument as above with reverse time shows that $u$ is microlocally $H^k$ outside $E_+^*$ over $\mc{U}$. Since $E_-^*\cap E_+^*=0$, we have thus shown that $u|_{\mc{U}}$ is in $H^k(\mc{U})$. 
\end{proof}

\subsection{Injectivity of the X-ray transforms: proof of Theorem \ref{th4}}

For each $m\in \nn_0$, we consider the map
\begin{align}\label{e:pi_m}
 \pi_m:SM\to\otimes_S^m TM,
 \qquad
 \pi_m(q,v)=v^{\otimes m}=v\otimes...\otimes v.
\end{align}
Notice that for $m=0$ this map is simply the base projection of the unit tangent bundle $\pi_0:SM\to M$.  The X-ray transform on $m$-tensors, which we defined in~\eqref{e:I_m}, is given by the composition $I_m:=I\circ\pi_m^*:C^0(M; \otimes_S^m T^*M)\to L^1(\pl_-SM\cup \pl_0SM,d\mu_\nu)$.

We now recall a theorem due to Sharatfudinov \cite[Theorem~1.3]{Sharafutdinov:2002ss} on the boundary determination of tensors in the kernel of $I_m$; since our setting is slightly more general than the one of \cite{Sharafutdinov:2002ss}, we will add a detailed proof for the reader's convenience. We first recall a few well known properties of Jacobi fields in compact Riemannian manifolds $(M,g)$ without conjugate points (see, e.g.,  \cite{Carmo:1992ye}). Given a geodesic $\gamma:[0,\ell]\to M$ joining $x:=\gamma(0)$ and $y:=\gamma(\ell)$, a Jacobi field $J(t)$ along $\gamma$ is completely determined by its values $J(0)$ and $J(\ell)$. Namely, there is a linear isomorphism of $T_xM\oplus T_yM$ onto the space of Jacobi fields along $\gamma$, given by $(J(0),J(\ell))\mapsto J$. This isomorphism depends smoothly on $\gamma$. In particular, fixed $\ell>0$, there exists a constant $c_{\ell}>0$ such that, for each geodesic $\gamma:[0,\ell]\to M$ of length at most $2\ell$ and for every Jacobi field $J$ along $\gamma$, we have
\begin{align*}
\|J(t)\|_g + \|\nabla_t J(t)\|_g \leq c_\ell \big( \|J(0)\|_g + \|J(\ell)\|_g \big),
\qquad
\forall t\in[0,\ell],
\end{align*}
where $\nabla_t$ denotes the Levi-Civita covariant derivative. Moreover, $g(J,\dot\gamma)\equiv0$ provided $g(J(0),\dot\gamma(0))=g(J(\ell),\dot\gamma(\ell))=0$.
\begin{lem}\label{lemsharaf}
Assume that $(M,g)$  has no conjugate points, 
and that $\mu_\nu(\Gamma_-\cap \pl SM)=0$.
If $f\in C^{m+1}(M;\otimes_S^m T^*M)$ satisfies $I_mf=0$, then there exists $p\in C^2(M;\otimes_S^{m-1}T^*M)$ such that $p|_{\pl M}\equiv0$ and $(f-Dp)|_{\pl M}\equiv 0$. In particular, if $m=0$, then $f|_{\pl M}=0$.
\end{lem}

\begin{proof}
The inner product with the inward pointing unit normal vector $\nu$ to $\partial M$ is the map
\begin{align*}
 \iota_\nu:C^1(M;\otimes_S^m T^*M)\to C^1(\pl M;\otimes_S^{m-1} T^* M|_{\partial M}),
 \qquad
 \iota_\nu f=f(\nu,\cdot,...,\cdot).
\end{align*}
It is well known (see e.g.\ \cite[Lemma 2.2.]{Sharafutdinov:2002ss}) that for every $f\in C^{m+1}(M;\otimes_S^m T^*M)$ there exists $p\in C^2(M;\otimes_S^{m-1} T^*M)$ such that $p|_{\partial M}\equiv0$ and $\iota_{\nu}(f-Dp)\equiv0$. Notice that $I_m(f-Dp)=I_m(f)$. Therefore, it is enough to prove that each tensor $f\in C^1(M;\otimes_S^m T^*M)$ such that $I_mf=0$ and $\iota_\nu f\equiv0$ must satisfy $f_{x_0}(v_0,...,v_0)=0$ for all $(x_0,v_0)\in \pl SM$. From now on, let us consider one such $f$.

Equation~\eqref{e:measure_at_boundary} readily implies that, when restricted to $\pl_-SM$, the measure $\mu_H$ induced by $H=G|_{\pl SM}$ is absolutely continuous with respect to $\mu_\nu$. In particular, since $\mu_\nu(\Gamma_-\cap \pl_- SM)=0$, we also have $\mu_H(\Gamma_-\cap \pl_- SM)=0$. This implies that the  $\pl_-SM\setminus \Gamma_-$ is dense in $\pl_-SM$, and therefore is dense in $\pl_-SM\cup \pl_0SM$ as well. We denote by $d:SM\times SM\to[0,\infty)$ the distance induced by $G$ on $SM$. The above density implies that, for each $(x_0,v_0)\in \pl_0SM$ and  $\epsilon\in(0,1)$, 
there exists $(x,v)\in \pl_-SM\setminus \Gamma_-$ such that $d((x,v),(x_0,v_0))<\epsilon$. To prove the lemma, it is enough to show that $|f_x(v,...,v)|\leq c\epsilon$ for some  $c>0$ independent of $(x,v)$.

Let $\gamma:[0,\ell]\to M$ be the unit speed geodesic such that $\gamma(0)=x$, $\dot\gamma(0)=v$, $y:=\gamma(\ell)\in\partial M$, and $\gamma(t)\in M^\circ$ for all $t\in(0,\ell)$. For notational convenience, up to replacing $g$ and $v$ by $\ell^{-2}g$ and $\ell v$ respectively, we can assume that $\ell=1$. Consider a smooth curve $s\mapsto w_s$ in $T_yM$ such that $w_0=-\dot\gamma(1)$. Since $y\in \pl M$, the geodesic $t\mapsto\exp_y(t w_s)$ is well defined at least for short positive time whenever $g(w_s,\nu_y)> 0$. Since the geodesic $\gamma(t):=\exp_y((1-t)w_0)$ intersects $\partial M$ transversely at $t=0$, and since $g(w_0,\nu_y)\geq 0$, by the implicit function theorem we can find a smooth curve $s\mapsto w_s$ as above such that:
\begin{itemize}
\item $g(w_s,\nu_y)> 0$ and $x_s:=\exp_y(w_s)\in\partial M$ for $s\neq0$,
\item $v_0':=\partial_s x_s|_{s=0}\in S_x\partial M$ satisfies $d((x,v_0'),(x_0,\delta v_0))<\epsilon$ for some $\delta\in\{1,-1\}$.
\end{itemize}
Up to replacing $v_0$ by $-v_0$, we can assume that $\delta=1$, and thus  $d((x,v_0'),(x,v))<2\epsilon$.

Now, we consider the family of geodesics $\gamma_s:[0,1]\to M$, $\gamma_s(t):=\exp_y((1-t)w_s)$, and the Jacobi field along $\gamma=\gamma_0$ given by
$J(t)
=
\partial_s \gamma_s(t) |_{s=0}$, 
which satisfies $J(0)=v_0'$ and $J(1)=0$. Notice that we can uniquely write $v_0'=\lambda v+n$, where $g(n,v)=0$, $\lambda\in(1-2\epsilon^2,1]$, and $\|n\|_g< 2\epsilon$. According to this splitting, we can write $J(t)=\lambda T(t) + N(t)$, where $T(0)=v=\dot\gamma(0)$, $N(0)=n$, and $T(1)=N(1)=0$. This readily implies that $g(N,\dot\gamma)\equiv0$, $\|N\|_g+\|\nabla_tN\|_g\leq c_1 \|n\|_g\leq c_1 2\epsilon$, and $T(t)=(1-t)\dot\gamma(t)$.

Since $f\in\ker(I_m)$, for each $s$ we have
\begin{align*}
\int_0^1 f_{\gamma_s(t)}(\dot\gamma_s(t),...,\dot\gamma_s(t))\,d t =0.
\end{align*}
By differentiating the above equality with respect to $s$ at $s=0$, we obtain
\begin{align}
\nonumber
 0 & = \int_0^1 \Big(\nnabla f_{\gamma(t)}(\partial_s\gamma_s(t)|_{s=0},\dot\gamma(t),...,\dot\gamma(t)) + m\,f_{\gamma(t)}(\nnabla_s\dot\gamma_s(t)|_{s=0},\dot\gamma(t),...,\dot\gamma(t))\Big)\,d t \\
\nonumber  & = \int_0^1 \Big(\nnabla f_{\gamma(t)}(J(t),\dot\gamma(t),...,\dot\gamma(t)) + m\,f_{\gamma(t)}(\nabla_t J(t),\dot\gamma(t),...,\dot\gamma(t))\Big)\,d t \\
 \label{e:integral_1}
  & = \int_0^1 \Big( \lambda(1-t)\nnabla f_{\gamma(t)}(\dot\gamma(t),...,\dot\gamma(t)) -  m\lambda f_{\gamma(t)}(\dot\gamma(t),...,\dot\gamma(t)) \Big)\,d t\\
 \label{e:integral_2} & \quad + \int_0^1 \Big( \nnabla f_{\gamma(t)}(N(t),\dot\gamma(t),...,\dot\gamma(t)) + m\, f_{\gamma(t)}(\nabla_t N(t),\dot\gamma(t),...,\dot\gamma(t))\Big)\,d t.
\end{align}
Here, $\nabla:C^1(M;\otimes^{m}T^*M)\to C^0(M;\otimes^{m+1}T^*M)$ denotes the Levi-Civita covariant derivative on tensors, whereas as before $\nabla_t$ denotes the Levi-Civita covariant derivative on vector fields along $\gamma$. The integral in~\eqref{e:integral_1} is equal to
\begin{align*}
&\int_0^1 \Big( \lambda(1-t)\nnabla f_{\gamma(t)}(\dot\gamma(t),...,\dot\gamma(t)) -  m\lambda f_{\gamma(t)}(\dot\gamma(t),...,\dot\gamma(t)) \Big)\,d t\\
&=
\int_0^1  \lambda(1-t)\nnabla f_{\gamma(t)}(\dot\gamma(t),...,\dot\gamma(t))d t=
\int_0^1  \lambda(1-t)\tfrac{d}{d t}f_{\gamma(t)}(\dot\gamma(t),...,\dot\gamma(t))d t
\\ 
&=
-\lambda f_x(v,...,v)
+
\int_0^1  \lambda\, f_{\gamma(t)}(\dot\gamma(t),...,\dot\gamma(t))d t=
-\lambda f_x(v,...,v).
\end{align*}
The absolute value of the integral in~\eqref{e:integral_2} has norm bounded 
above by 
\[\begin{gathered}
 \|\nabla f \|_{L^\infty} \|N\|_{L^\infty} + m\|f\|_{L^\infty} \|\nabla_t N\|_{L^\infty}\leq  \big(\|\nabla f\|_{L^\infty}  + m\|f\|_{L^\infty}\big) c_1 2\epsilon.
\end{gathered}\]
All together, we obtain the desired estimate
\begin{align*}
|f_x(v,...,v)|
\leq
\lambda^{-1}  \big(\|\nabla f\|_{L^\infty}  + m\|f\|_{L^\infty}\big) c_1 2\epsilon
\leq
(1-2\epsilon^2)^{-1}  \big(\|\nabla f\|_{L^\infty}  + m\|f\|_{L^\infty}\big) c_1 2\epsilon.
\end{align*}
\end{proof}

\begin{proof}[Proof of Theorem \ref{th4}]
Let us recall the Pestov identity on $SM$ as written in \cite{Paternain:2015ud}. We decompose the gradient a function $u\in C^1(SM)$ with respect to the Sasaki metric $G$ as
\begin{align*}
 \nabla u
 =
 (Xu)X +\nablav u+\nablah u,
\end{align*}
where $\nablav u\in \mc{V}$ and $\nablah u\in \mc{H}$. We recall that we can identify both the vertical and horizontal subbundles $\mc{V}$ and $\mc{H}$ with the vector bundle $\mc{Z}$. Given a $C^1$-section $w$ of $\mc{Z}$, we write $Xw$ for its covariant derivative along geodesic flow lines, i.e.
\begin{align*}
Xw(x,v)= \nabla_t ( w\circ\varphi_t(x,v))|_{t=0}.
\end{align*}
We denote by $\mc{R}_{x,v}:\mc{Z}_{x,v}\to \mc{Z}_{x,v}$ the operator defined by 
$\mc{R}_{x,v}w=\mc{R}(w,v)v$ where $\mc{R}$ is the Riemann curvature tensor of $g$. Then Pestov identity is \begin{equation}\label{Pestov}
\|\nablav Xu\|_{L^2}^2-\|X\nablav u\|_{L^2}^2 - (n-1)\|Xu\|^2_{L^2}+\cjg \mc{R}\nablav u,\nablav u\cjd_{L^2}=0\end{equation}
for all $u\in C^\infty(M)$ with $u|_{\pl M}=0$. 

In order to prove (i), we use Lemma \ref{lemsharaf} together with Proposition \ref{Livsic}: there is $u\in H_0^1(SM)$ such that $Xu=\pi_0^*f$ and $u|_{\pl SM}=0$. In particular,  $\nablav Xu=0\in H^1(SM)$ and $Xu\in H_0^1(SM)$.
Thus, by \cite[Lemma E.47]{Dyatlov:2017tx}, there is a sequence of smooth functions $u_j$ vanishing at $\pl SM$ such that $u_j\to u$ in $H_0^1(SM)$ 
such that $Xu_j\to Xu=\pi_0^*f$ in $H^1(SM)$. Using that 
$X\nablav u=\nablav Xu-\nablah  u=-\nablah  u\in L^2$, we also have the convergence in $L^2$
\[X\nablav u_j=\nablav Xu_j-\nablah  u_j\to  \nablav Xu-\nablah  u=-\nablah  u\] 
and therefore taking the limit for $j\to \infty$ in \eqref{Pestov} (applied to $u_j$), we obtain
\begin{equation}\label{pestovfinal}
0= \|\nablah u\|_{L^2}^2+(n-1)\|Xu\|^2_{L^2}-\cjg \mc{R}\nablav u,\nablav u\cjd_{L^2}.\end{equation}
We can now use that $f\in C^1(M)$ to deduce that $u\in C^1(SM\setminus (\mc{K}\cup \mc{G}))$  by Proposition \ref{Livsic}. This implies that $Y:=\nablav u$ is $C^0$ on each non-trapped and non-glancing geodesic in $SM$, and since $XY=-\nablah u$ is also $C^0$ on such geodesic, we conclude that $Y$ is a $C^1$ vector field along such geodesic. The bad set $\mc{B}:=\Gamma_+\cup \Gamma_-\cup 
\mc{G}$ has measure $\mu(\mc{B})=0$, and since $Y|_{\pl SM}=0$ we get using Santalo formula \eqref{santalo} and the fact that 
$\mu(\Gamma_\pm)=\mu(\mc{G}_\pm)=0$ 
\[\|XY\|_{L^2}^2-\cjg \mc{R}Y,Y\cjd_{L^2}=
\int_{\pl_-SM\setminus \mc{B}}I(F(Y))d\mu_\nu
\]
where $F(Y):=|XY|_g^2-\cjg \mc{R}Y,Y\cjd_g\in C^0(SM\setminus \mc{B})$ and $I$ is the X-ray transform. Notice that 
$I(F(Y))(x,v)$ is the index form of the vector field $Y$ along the geodesic $t\mapsto \exp_x(tv)$. 
If $g$ has no pair of conjugate points in $M$, then $I(F(Y))\geq 0$ everywhere on $\pl_-SM$, thus $\|\pi_0^*f\|_{L^2}^2=\|Xu\|^2_{L^2}=0$ by using \eqref{pestovfinal}.

To show (ii), we use a similar argument: by Lemma \ref{lemsharaf} we first reduce to 
the case $\iota_{\nu}f=0$ and $f|_{\pl M}=0$ and we apply Proposition \ref{Livsic} to find 
$u\in H_0^1(SM)$ so that $Xu=\pi_1^*f$.
Since $f$ is a $1$-form, we have 
\[\|\nablav Xu\|^2_{L^2}=(n-1)\|Xu\|_{L^2}^2\] 
and thus from Pestov identity \eqref{Pestov} and an appoximation argument as above,
\begin{equation}\label{conseqpestov}
0=\|X\nablav u\|_{L^2}^2 -\cjg \mc{R}\nablav u,\nablav u\cjd_{L^2}
\end{equation}
We have $Y:=\nablav u\in L^2(SM)\cap C^0(SM\setminus \mc{B})$, since $u\in H^1(SM)\cap C^1(SM\setminus \mc{B})$. Moreover, $Y$ 
satisfies $XY=\nablav \pi_1^*f-\nablah u\in C^0(SM\setminus \mc{B})\cap L^2(SM)$ and $Y|_{\pl SM}=0$. Thus the restriction of $Y$ along each non-trapped/glancing geodesic 
is a $C^1$-vector field. By \eqref{conseqpestov} and Santalo's formula, we get that the index form $I(F(Y))$ vanishes 
along all non-trapped and non-glancing geodesics. Since $g$ has no conjugate points, the index form along a non-trapped geodesic $\gamma$ is positive definite on the space of $C^1$ vector field 
vanishing at $\pl M$, thus $Y=0$ along $\gamma$, and thus $Y=0$ on $SM$.
 Then $u=\pi_0^*q$ for some $q\in H_0^1(M)$ and $f=dq$ (since $X\pi_0^*q=\pi_1^*dq$).
 In particular we also get $q\in C^2(M)$ since $f\in C^1(M;T^*M)$. 
 
The case (iii) of tensors is proved as in \cite[Sections 9 and 11]{Paternain:2015ud}: once we 
know that $Xu=\pi_m^*f$ for some $u\in H_0^1(SM)$ we can use Pestov identity the same way as in \cite{Paternain:2015ud} using that the flow is $1$-controlled in the terminology of \cite{Paternain:2015ud} when the curvature is non-positive.
\end{proof}

\begin{rem}
Note that the assumption of non-positive curvature can be relaxed to an $\alpha$-controlled condition as in \cite{Paternain:2015ud}.
\end{rem}

\section{Invariant distributions in the non-trapping case}

Throughout this section, we assume that the compact Riemannian manifold with boundary $(M,g)$ is non-trapping and has no conjugate points. We recall that we can see $M$ as a compact subset of a closed Riemannian manifold $(N,g)$ with $\dim(N)=\dim(M)$. By Lemma \ref{l:no_conjugate_points}, there exists two sufficiently small compact neighborhoods $M_1$ and $M_2$ of $M$ with smooth boundary, such that $M_1\subset M_2^\circ$ and $(M_2,g)$ is non-trapping and without conjugate points. For $i=1,2$, we set
$\mc{M}_i:=SM_i$, and we consider the exit time function $\tau^{\pm}_{\mc{M}_i}$ defined as in \eqref{exittime}. We will also need to consider the \emph{hitting time functions}
$t^{+}_{\mc{M}_i}:\mc{M}_i \to [0,\infty)$ and $t^{-}_{\mc{M}_i}:\mc{M}_i \to (-\infty,0]$, given by 
\begin{equation}\label{hitting}
t_{\mc{M}_i}^\pm(y):= \pm \inf\big\{t>0\ \big|\ \varphi_{\pm t}(y)\not\in \mc{M}_i^\circ\big\}.
\end{equation}
The X-ray transform on $\mc{M}_i$, which is defined as in \eqref{defXray}, will be denoted by 
\[I^{\mc{M}_i}:  C^\infty(\mc{M}_i)\to L^\infty(\pl_{-}\mc{M}_i\cup \pl_0\mc{M}_i).\]
As in \eqref{Gpm}, we denote by $\mc{G}^1_\pm\subset \mc{M}_1$ the flowout of $\pl_0\mc{M}_1$, and we set $\mc{G}_*^1:=\mc{G}^1_+\cup \mc{G}^1_-$. 

\subsection{Localized X-ray transform near the glancing region}

For each $y\in \mc{G}_*^1$, there exists $\delta_\pm(y)>0$ such that 
\[ y^\pm:=\varphi_{t^\pm_{\mc{M}_2}(y)\mp \delta_\pm(y)}(y)\in \mc{M}_2^\circ \setminus\mc{M}_1.
 \]
There is an open neighborhood $S^\pm_{y}$ of $y^{\pm}$ inside a hypersurface orthogonal to $X$ (with respect to Sasaki metric) such that $S^{\pm}_y\subset \mc{M}_2^\circ \setminus \mc{M}_1$.
Consider the flow line $\gamma_y\subset \mc{M}_2^\circ$ passing through $y$ with extremities $[y^-,y^+]$, 
then there is a small flow-box neighborhood $\mc{G}^1_*(y)$ of $\gamma_y$ (i.e. a neighborhood made of flow lines) contained in $\mc{M}_2^\circ$ and 
intersecting $S^{\pm}_y$ in two open sets $B^\pm_y\subset S^{\pm}_y$, in a way that 
$\mc{G}^1_*(y)$ is a forward flowout of $B_{y}^-$.
By compactness, we can extract a finite number of trajectories $(\gamma_{y_j})_{j=1,\dots J}$ for some $y_j\in \mc{M}_1$ such that $\mc{G}^1_*\subset \cup_{j\leq J}\mc{G}^1_*(y_j)$.
For each $y_j$, there exists a non-negative $\chi_j\in C_c^\infty(B^-_{y_j})$, positive near the point 
$y^-_j:=\gamma_{y_j}\cap B^-_{y_j}$, satisfying that if $\til{\chi}_{j}$ is the flow-invariant function 
in $\mc{G}^1_*(y_j)$ with boundary condition $\til{\chi}_j|_{B^-_{y_j}}=\chi_j$, 
there exists an open set $U_{\mc{G}^1_*}\subset \mc{M}_2^\circ $ containing $\mc{G}^1_*$ such that  
$\sum_{j\leq J}(\til{\chi}_j)^2=1$ on $U_{\mc{G}^1_*}$,   
and $U_{\mc{G}^1_*}\cap\mc{M}_1$ is flow-invariant in the sense that 
\[y\in U_{\mc{G}^1_*}\cap\mc{M}_1 \Longrightarrow\varphi_{[\tau^-_{\mc{M}_1}(y),\tau^+_{\mc{M}_1}(y)]}(y)\subset U_{\mc{G}_*^1}\cap\mc{M}_1.\]
Let $\theta\in C_c^\infty(M^\circ_2;[0,1])$ be equal to $1$ on a small neighborhood of $M$, and with $\supp(\theta)\subset M_1^\circ$. 
We also write $\theta$ for its pull-back $\pi_0^*\theta$ to $\mc{M}_2$. 

We first define the weighted X-ray transform 
$I_0^{\chi_j}:=I^{\chi_j}\pi_0^*$ as the map
\[I^{\chi_j}: C^\infty(\mc{M}_1)\to C^0(S^-_{y_j}) , \quad 
I^{\chi_j}f(y)=\chi_j(y)\int_{0}^{\tau^+_j(y)}f(\varphi_t(y))dt, 
\] 
where $\tau^+_{j}(y)$ is the first 
time so that $\varphi_{\tau^+_{j}(y)}(y)\in S^+_{y_j}$ and $f$ is extended by $0$ outside $\mc{M}_1$. We put 
the measure $ds_j:=\iota_{S_j}^*i_X\mu$ on $S^{-}_{y_j}$ where $\iota_{S_j}:S^{-}_{y_j}\to \mc{M}_2$ is the inclusion map. This induces a natural $L^2(S^{-}_{y_j})$ space and it is direct to see that $I^{\chi_j}:L^2(\mc{M}_1)\to L^2(S^{-}_{y_j})$ is bounded. 
 Then, we define the operator acting on $L^2(M_1)$ 
\[ P_1:=  \sum_{j=1}^J \theta (I_0^{\chi_j})^*I_0^{\chi_j}\theta\]
where $I_0^{\chi_j}:=I^{\chi_j}\pi_0^*$
and the adjoint is with respect to the measures $\mu$ for $\mc{M}_1$ and $ds_j$ for $S^{-}_{y_j}$.
Notice that $\tau^+_{j}$ is a smooth function on 
$\mc{G}^1_*(y_j)\cap \mc{M}_1$ due to the transversality of $X$ with $S^+_{y_j}$.
\begin{lem}\label{P1PDO} 
The operator $P_1$ is a pseudo-differential operator or order $-1$ on $M_1$, thus on $M_2$ as well.
There is $c>0$ such that
the principal symbol $\sigma(P_1)$ satisfies
\[ \sigma(P_1)(x,\xi)>c|\xi|^{-1}\]
for all $(x,\xi)\in T^*M_1$ for which there is $v\in S_xM_1$ such that  $(x,v)\in U_{\mc{G}^1_*}$
and $\xi(v)=0$.
\end{lem}
\begin{proof} The proof is a combination of arguments 
in \cite[Lemma 3]{Stefanov:2008bd} and \cite{Guillarmou:2017if}. A direct computation gives the expression 
\[ (I_0^{\chi_j})^*I_0^{\chi_j}f(x)=\int_{S_xM_1} \int_{0}^{\tau_j^+(x,v)}f(\pi_0(\varphi_t(x,v)))(\til{\chi}_j^2(x,v)+\til{\chi}_j^2(x,-v))dtd\Omega_x(v) \]
where $\til{\chi}_j$ is the smooth function on $SM_1$ 
satisfying $X\til{\chi}_j=0$ and $\til{\chi}_j|_{S_{y_j}^-}=\chi_j$ (here $d\Omega_x$ is the natural measure on $S_xM_1$). The $t$-integral can be written as 
\[\int_{0}^{\tau_j^+(x,v)}f(\pi_0(\varphi_t(x,v)))dt =\int_{0}^\eps \psi(t)f(\exp_{x}(tv))dt+
\int_{0}^{\tau_j^+(x,v)}(1-\psi(t))f(\exp_{x}(tv))dt 
\] 
with $\psi\in C_c^\infty(-\eps,\eps)$ equal to $1$ near in $[-\eps/2,\eps/2]$, for some small $\eps>0$. We can write
\[\int_{S_xM_1} \int_{0}^{\tau_j^+(x,v)}(1-\psi(t))f(\pi_0(\varphi_{t}(x,v)))F_j(x,v)dtdv=
\int_{M_1}K_1(x,x')f(x'){\rm dv}_g(x') \]
with $K_1\in \mc{D}'(M_1^\circ \x M^\circ_1)$ and $F_j(x,v):=\til{\chi}_j^2(x,v)+\til{\chi}_j^2(x,-v)$. We claim that $K_1\in C^\infty(M^\circ_1\x M^\circ_1)$.
To prove this, we first recall that the pull-back operator by $\varphi_{t}$, viewed as a map $C^\infty(\mc{M}_1)\to C^\infty(\rr\x \mc{M}_1)$  is a Fourier integral operator and by \cite[Theorem 8.2.4]{Hormander:1983dw} the wavefront set of its Schwartz kernel is  contained in
\[ 
\big\{ (t,y,\varphi_{t}(y);-\eta(X(\varphi_{t}(y))),-d\varphi_{t}(y)^T\eta,\eta)\ \big|\ t\geq 0, y\in \mc{M}_2^\circ, \eta\in T^*_{\varphi_{t}(y)}\mc{M}_2^\circ\big\}.\]
By the pushforward rule for wavefront sets \cite[Theorem 8.2.12]{Hormander:1983dw}, we obtain
that ${\rm WF}(K_1|_{M_1^\circ \x M_1^\circ})\subset W$ with 
\[\begin{split}
W:=\big\{ & \pi_0(y),\pi_0(y'),\xi,\xi')\in T^*(M_1 \x M_1)\ \big|\
\exists t>\eps/2, y'=\varphi_t(y),
  \exists \, \eta \in T^*\mc{M}_1,\\ 
            & \,\, \eta(X(y'))=0 , \, d\pi_0(y')^T\xi'=\eta, \,\, d\pi_0(y)^T\xi=-d\varphi_t(y)^{T}\eta,
             \,\, y\in \supp(\til{\chi}_j)\big\}.
\end{split}\]
Assuming that the geodesic flow has no conjugate points in $M_1$, i.e. $d\varphi_t(y)\mc{V}\cap \mc{V}=0$ for all $t\in(0, \tau^+_{\mc{M}_1}]$, it is direct to see that $W\setminus \{\xi=\xi'=0\}=\varnothing$, showing that $K_1\in C^\infty(M_1^\circ\x M_1^\circ)$.
We next analyse the small time integral and write 
\[ \int_{S_xM_1} \int_{0}^{\eps}\psi(t)f(\exp_{x}(tv))F_j(x,v)dtd\Omega_x(v)=
\int_{M_1}K_2(x,x')f(x'){\rm dvol}_g(x')\]
for some integral kernel $K_2$ on $M_1$.
Changing coordinates to $x'=\exp_{x}(tv)=\pi_0(\varphi_t(x,v))$ (using that the exponential map is a diffeomorphism near $0$), we have $t=d_g(x,x')$ and $v(x,x')=\exp_x^{-1}(x')/d_g(x,x')$
and we can write on $M_1\x M_1$ 
\[ K_2(x,x')= F_j(x,v(x,x'))\frac{\psi(d_g(x,x'))}{d_g(x,x')^{n-1}}J(x,x')\sqrt{\det(g_x)}\]
where $J(x,x')$ is the Jacobian of the map $tv=\exp^{-1}_{x}(x')$ (here we use the Riemannian measure on $T_xM$). The function $J(x,x')$ is smooth and $J(x,x)=1$,
thus $d_g^{n-1}K_2$ is smooth in polar coordinates around the diagonal, which implies that 
$K_2\in \Psi^{-1}(M_1)$ is a classical pseudo-differential operator of order $-1$.
Its local symbol is given by $(x,\xi)\mapsto \mc{F}_{z\to \xi}(K_2(x,x-z))$ and,
 its principal symbol is (here $|z|^2_{g_x}=\sum_{i,j}g_{ij}(x)z_iz_j$)
\[ \sigma_j(x,\xi)= \sqrt{\det g_x}\int_{\rr^n}e^{i\xi.z}|z|_{g_x}^{-n+1}F_j\Big(x,\frac{z}{|z|_{g_x}}\Big)dz\]
Using that $F_j(x,v)=F_j(x,-v)$ and using polar coordinates $z=r\omega$, we obtain
\[\sigma_j(x,\xi)=2\pi |\xi|_{g^{-1}_x}^{-1}\int_{\{|v|_{g_x}=1\, |\, \xi(v)=0\}}  \!\!\!\!\!\!\!\!\!\!\!\!\!\!\!\!\!\!\!\!  F_j(x,v)d\Theta_x(v)\]
which satisfies  $\sigma_j(x,\xi)\geq c_0|\xi|_{g^{-1}}^{-1}$ for some $c_0>0$ at those $\xi\in T^*M_1$ such that there exists $v\in S_xM_1$ with $\xi(v)=0$ and $F_j(x,v)>c_1$ for some $c_1>0$ ($c_0$ depending on $c_1$). Here $d\Theta_x(v)$ is a natural measure obtained from $d\Omega_x(v)$  on the submanifold $\{|v|_{g_x}=1\ |\ \xi(v)=0\}\subset S_xM$. 
Thus there is $c>0$ such that for all 
$(x,v)\in U_{\mc{G}^1_*}$, if $\xi(v)=0$ we obtain that the principal symbol $\sigma(P_1)$ of $P_1$ satisfies 
$\sigma(P_1)(x,\xi)>c|\xi|_{g^{-1}}^{-1}$.
\end{proof}

\subsection{X-ray transform outside the glancing region}
Let $\chi_0\in C_c^\infty(\pl_-\mc{M}_1; [0,1])$ such that 
$\chi_0=1$ on $\pl_-\mc{M}_1\setminus U_{\mc{G}^1_*}$ and $\chi_0=0$ 
near $\mc{G}^1_*$, and let $\til{\chi}_0$ be the function on $\mc{M}_1$ that is flow-invariant  such that $\til{\chi}_0|_{\pl_-\mc{M}_1}=\chi_0$. We define the operator 
\[ P_2:= \theta {(I_0^{M_1})}^*\chi_0^2I_0^{M_1}\theta\]
acting on $C_c^\infty(M_2^\circ)$, where $I_0^{M_1}=I^{\mc{M}_1}\pi_0^*$ is the $X$-ray transform on functions on $M_1$.
\begin{lem}\label{P2PDO}
The operator $P_2$ is a pseudo-differential operator or order $-1$ on $M_2$.
There is $c>0$ such that the principal symbol $\sigma(P_2)$ satisfies
\[ \sigma(P_2)(x,\xi)>c|\xi|^{-1}\]
for all $(x,\xi)\in T^*M_1$ for which there is $v\in S_xM_1$ such that  $(x,v)\notin U_{\mc{G}_*^1}$ and $\xi(v)=0$.
\end{lem}
\begin{proof} The proof is exactly the same as for Lemma \ref{P1PDO}
\end{proof}

\subsection{Surjectivity}
With $\theta\in C_c^\infty(M^\circ_2;[0,1])$ as above (supported in $M_1^\circ$), let us define the self-adjoint operator on the closed manifold $N$
\[ P_3:=(1-\theta)(1+\Delta_g)^{-1/2}(1-\theta).\]
The operator $P_3$ is an elliptic pseudo-differential operator of 
order $-1$ with principal symbol
\begin{equation}\label{P3PDO}
\sigma(P_3)(x,\xi)=(1-\theta)^2(x)|\xi|_{g^{-1}}^{-1}.
\end{equation}
Next, consider the following operator on the closed manifold $N$
\[ P:= P_1+P_2+P_3\]
Lemmas \ref{P1PDO} and \ref{P2PDO} readily imply the following corollary.
\begin{cor}
The operator $P$ is an elliptic pseudo-differential operator of order $-1$ on $N$, and is therefore Fredholm as a map $H^s(N)\to H^{s+1}(N)$ for each $s\in \rr$.
\end{cor} 
We now show that $P$ is injective. We denote by ${\pi_0}_*$ the operator that integrates a function on $SN$ (or on the submanifolds $M_2$, $M_1$, and $M$) in the fibers with respect to the measure induced by the Riemannian metric $g$.
\begin{prop}\label{Pinjective}
For each $f\in H^k(M)$ 
with $k\geq 1$, there exists 
$w\in H^{k-1}(SM_1)$ such that $Xw|_{SM}=0$ and ${\pi_0}_*w=f$ in $M$. 
If $f\in C^\infty(M)$, then $w$ can be chosen in $C^\infty(SM)$.
\end{prop}
\begin{proof} Since $P_j\geq 0$ for each $j=1,2,3$, we have that for each $f\in \ker P\cap L^2(N)$, 
\[ \cjg P_1f,f\cjd_{L^2}=\cjg P_2f,f\cjd_{L^2}=\cjg P_3f,f\cjd_{L^2}=0.\] 
and by ellpiticity of $P$ we also have $f\in C^\infty(N)$. Using $\ker (1+\Delta_g)^{-1}=0$, we obtain 
\begin{equation}\label{identXrays}
(1-\theta)f=0, \quad \chi_0I_0^{M_1}(\theta f)=0 ,\quad \forall j\leq J,\,\,  I_0^{\chi_j}(\theta f)=0,
\end{equation} 
which clearly implies $f\in C_c^\infty(M_1^\circ)$, $\chi_0I_0^{M_1}f=0$ and $I_0^{\chi_j}f=0$ for each $j\leq J$.
We claim that this implies that $I_0^{M_2}f=0$ where $I_0^{M_2}:=I^{\mc{M}_2}\pi_0^*$ is the $X$-ray transform on functions on $M_2$. Indeed, take a point $y\in \pl_-\mc{M}_2\cup \pl_0\mc{M}_2$ and let $\gamma_y$ be the integral curve of $X$ in $\mc{M}_2$ passing by $y$. 
If $\gamma_y\cap U_{\mc{G}_*^1}=\varnothing$, then $\gamma_y\cap \mc{M}_1$ splits into 
a finite family of connected components $\gamma_y^k$ for $k=1,\dots k_0$ and 
\[ I_0^{M_2}f(y)=\sum_{k=1}^{k_0} I_0^{M_1}f(y'_k)\]
where $y'_k\in \pl_-\mc{M}_1\cup \pl_0\mc{M}_1$ are the extremities of $\gamma_{y}^k$.
But  $\chi_0(y_k)=1$ since $y_k\notin U_{\mc{G}^1_*}$, thus we get $I_0^{M_2}f(y)=0$.
If now $\gamma_y\cap U_{\mc{G}^1_*}\not =\varnothing$,  
$\gamma_y\setminus (\gamma_y\cap U_{\mc{G}^1_*})$ intersects $\mc{M}_1$ 
into finitely many connected components denoted by $\gamma_y^{k}$ for $k=1,\dots, k_0$ for some $k_0\in\nn$. Since $f$ is supported in $M_1$, we get 
\begin{equation}\label{decompI_0} 
I_0^{M_2}f(y)=\sum_{k=1}^{k_0} I_0^{M_1}f(y'_k)+ \int_{\gamma_y\cap U_{\mc{G}^1_*}}f
\end{equation}
where $y'_k\in \pl_-\mc{M}_1\cup \pl_0\mc{M}_1$ are the incoming extremities of $\gamma_{y}^k$. The first sum in \eqref{decompI_0} vanishes by \eqref{identXrays} while 
\[\int_{\gamma_y\cap U_{\mc{G}^1_*}}f=\sum_{j=1}^{J}I_0^{\chi_j}f(y''_j)=0\]
where $y''_j\in S^-_{y_j}$ are the incoming extremities of $\gamma_y\cap\mc{G}_{y_j}$.
We have thus proven that $I_0^{M_2}f=0$, and since $M_2$ has no pair of conjugate points and is non-trapping, we deduce by Theorem \ref{th4} that $f=0$, showing that $P$ is injective on $L^2(N)$.

Since $P:H^s(N)\to H^{s+1}(N)$ is Fredholm with index $0$ for each $s\in\rr$, we deduce that $P:H^s(N)\to H^{s+1}(N)$ is surjective. For $f\in H^k(M)$, we can extend it in $H^k_{\rm comp}(M_1^\circ)$ arbitrarily, then there exists a unique $u\in H^{k-1}(N)$ such that $Pu=f$ (if $f\in C^\infty$, then $u\in C^\infty$). Restricting this equality to $M$, we get $(P_1u+P_2u)=f$ in the region $\{\theta=1\}$, which means that in that same region
\[ {\pi_0}_*({(I^{M_1})}^*\chi_0^2I_0^{M_1}\theta u+\sum_{j=1}^J (I^{\chi_j})^*I_0^{\chi_j}\theta u)=f.\]
Set $w:={(I^{M_1})}^*\chi_0^2I_0^{M_1}\theta u+\sum_{j=1}^J (I^{\chi_j})^*I_0^{\chi_j}\theta u$. Now it is standard (e.g. \cite[Theorem 4.2.1]{Sharafutdinov:1994lh}) that $I_0^{\chi_j}u\in H^{k-1}(S_{y_j}^-)$, $\chi _0I_0^{M_1}u
\in H^{k-1}(\pl_-SM_1)$ and that 
\begin{equation}\label{regI*}
r_{SM_1}(I^{\chi_j})^{*}:H^{k-1}(S_{y_j}^{-})\to H^{k-1}(SM_1), \quad (I^{M_1})^*\chi_0:H^{k-1}(\pl_-SM_1)\to H^{k-1}(SM_1).\end{equation}
where $r_{SM_1}$ is the restriction to $SM_1$.
Thus $w$ is a $H^{k-1}(\mc{M}_1)$ function that satisfies $Xw=0$ and ${\pi_0}_*w=f$ in $M$, and it is smooth if $f$ was smooth.
 \end{proof}

\begin{rem}\label{assreg}
For each $f\in L^1(M)\cap\ker(I_0)$, we have $Pf=0$. Since $P$ is elliptic, we infer that $f\in C^\infty(M)$. This, together with Theorem~\ref{th4}, implies that $I_0$ is injective on $L^1(M)$ provided the compact Riemannian manifold $(M,g)$ is non-trapping and has no conjugate points.
\end{rem}

\subsection{The case of divergence-free $1$-forms}
We denote the adjoint of the pull-back $\pi_1^*$ by ${\pi_1}_*:\mc{D}'(SM^\circ)\to\mc{D}'(M^\circ;T^*M)$. We notice that $D=d$ on functions and thus $D^*=d^*$ is the divergence on $1$-forms.

\begin{prop}\label{surjI_1*}
For each $f\in H^k(M;T^*M)$ with $D^*f=0$ for $k\in 2\nn$, there exist 
$w\in H^{k-1}(SM_1)$ such that $Xw|_{SM}=0$ and ${\pi_1}_*w=f$ in $M$.
\end{prop}
\begin{proof}
We proceed as for the case of functions. We use the same notations as in the previous subsections. Let $P_1:=  \sum_{j=1}^J (I_1^{\chi_j})^*I_1^{\chi_j}$ where $I_1^{\chi_j}:=I^{\chi_j}{\pi_1}^*$, and let $P_2:= {(I_1^{M_1})}^*\chi_0^2I_1^{M_1}$ viewed both as operators on $M_1$. 
Notice that $D^*Pf=0$ on $M_1^\circ$ since $D^*(I_1^{M_1})^*=0$ and $D^*(I_1^{M_2})=0$ (by using that $I^{M_j}_1D=0$ on $H_0^1(M_j)$). 
The operator $P=P_1+P_2$ is a classical pseudo-differential of order $-1$ on $M_1^\circ$ and the symbol is calculated in the same way as we did for functions: we can use for example the same arguments as \cite[Lemma 3]{Stefanov:2008bd} (or  
\cite{Sharafutdinov:2005mw} in the case with no weight) for the computation of the principal symbol, combined with the inclusion of the cutoffs $\chi_j$ like for the case of functions. The principal symbol of $P_1$ is $\sigma(P_1)(x,\xi)=\sum_{j=1}^J\sigma_j(x,\xi)$ with $\sigma_j(x,\xi)$ being the $n\x n$ matrix in the basis $dx_1,\dots,dx_n$ of $T_x^*M_1$ 
\[(\sigma_j(x,\xi))_{k\ell}= \sqrt{\det g_x}\int_{\rr^n}e^{i\xi.z}|z|_{g_x}^{-n-1}
F_j\Big(x,\frac{z}{|z|_{g_x}}\Big)((g_xz)_k z_\ell)dz\]
where $(g_xz)_k:=\sum_{i=1}^ng_{ki}(x)z_i$. Using polar coordinates and the fact that  $F_j(x,v)=F_j(x,-v)$, we obtain like for the case of functions
\[(\sigma_j(x,\xi))_{k\ell}=2\pi |\xi|_{g^{-1}_x}^{-1}\int_{\{|v|_{g_x}=1\,|\, \xi(v)=0\}}
F_j(x,v)(g_xv)_kv_\ell d\Omega_x(v).\]
Since $F_j\geq 0$, we have for $\eta\in T_x^*M_1$, 
\[\cjg\sigma_j(x,\xi)\eta,\eta\cjd_{g^{-1}_x}=2\pi|\xi|_{g^{-1}_x}^{-1}\int_{\{|v|_{g_x}=1\,|\, \xi(v)=0\}}F_j(x,v)(\eta(v))^2d\Theta_x(v) \geq 0. 
\]
Let $\xi\in T_xM^*$ so that there is $v_0\in T_xM$ where $\chi_j(v_0)>0$ and $\xi(v_0)=0$. Then if $\eta\in T_xM^*$ satisfies $\sigma_j(x,\xi)\eta=0$, we get that $\eta(v)=0$ for $v$ near $v_0$ satisfying $\xi(v)=0$. If in addition $\cjg \xi,\eta\cjd_{g^{-1}_x}=0$, then we obtain $\eta=0$, 
showing injectivity of $\sigma_j(x,\xi)$ on the set $\xi^{\perp}$.
Summing this over $j$, we obtain that for those $\xi\in T^*M$ such that there exists $(x,v)\in U_{\mc{G}_*^1}$ so that $\xi(v)=0$, $\sigma(P_1)(x,\xi)\geq c_0|\xi|^{-1}_{g^{-1}_x}$ on $\xi^\perp$ for some $c_0>0$. The same argument works for $P_2$, i.e. $\sigma(P_2)(x,\xi)\geq c_1|\xi|^{-1}_{g^{-1}_x}$ on $\xi^\perp$ for some $c_1>0$ when there is $(x,v)\not \in U_{\mc{G}^1_*}$ so that $\xi(v)=0$. We conclude that $P$ is elliptic at each $(x,\xi)\in T^*M$ on the subspace 
$\{ \eta\in T^*_xM\ |\  \cjg \eta,\xi\cjd_{g_x^{-1}}=0\}$. Now, \cite{Sharafutdinov:2005mw} shows that there are some classical pseudo-differential operator 
$Q\in \Psi^{1}(M_1^\circ)$, $S\in \Psi^{-2}(M_1^\circ)$ and $R\in \Psi^{-\infty}(M_1^\circ)$ with compact support in a domain containing $M$ such that 
$QP={\rm Id}+DSD^*+R$ near $M$. 
Now we follow the proof of \cite[Lemma 2.2]{Dairbekov:2010km}: there is a  continuous extension map $E:\ker D^*|_{L^2(M)}\to \ker D^*|_{L^2(M_1)}$ restricting to $E:\ker D^*|_{H^k(M)}\to \ker D^*|_{H^k(M_1)}$  for any $k\in 2\nn$ large and, if 
$r_M:L^2(M_1^\circ)\to L^2(M)$ is the restriction map, we have $r_MPQ^*E={\rm Id}+r_MR^*E$ on $\ker D^*|_{L^2(M)}$ and its range is contained in $\ker D^*|_{L^2}$.
This implies that the range of ${\rm Id}+r_MR^*E$ is closed with finite codimension in $\ker D^*|_{L^2(M)}$, 
and the same holds on the Hilbert space $H^k_{D^*}:=\ker D^*|_{H^k(M)}$ (equipped with the norm $||f||^2_{H^k}=||(1+\Delta_g)^{k/2}f||^2_{L^2}$ if  $\Delta_g:=d^*d+dd^*$ is the Hodge Laplacian on $1$-forms). 
Then $r_MPQ^*E(H_{D^*}^k(M))$ has closed range in $\ker D^*|_{H^k(M)}$
with finite codimension and thus $r_MP:H^{k-1}(M_1)\to H^k_{D^*}$ has closed range with finite codimension. Let us prove that the adjoint 
$(r_MP)^*$ is injective. Let $f\in \ker (r_MP)^*\cap H^k_{D^*}$: one has for all 
$u\in H^{k-1}(M_1)$ (using $k/2\in \nn$ and $P^*=P$ on $L^2(M_1)$)
\[0=\cjg (1+\Delta_g)^{k/2}r_MPu,(1+\Delta_g)^{k/2}f\cjd_{L^2(M)}=
\cjg u, P(1+\Delta_g)^{k/2}e_M(1+\Delta_g)^{k/2}f\cjd_{L^2(\til{M})}
\]
where $e_M:L^2(M)\to L^2(M_1)$ is the inclusion.
Define $f':=(1+\Delta_g)^{k/2}e_M(1+\Delta_g)^{k/2}f\in H^{-k}_{\rm comp}(M_1^\circ)$, this solves $Pf'=0$ in $M_1^\circ$ and, since $D^*=d^*$ commute with $\Delta_g$, we get $D^*f'=0$ outside $\pl M$. By elliptic regularity and since $\supp(f')\subset M_1^\circ$, we have $f'\in C^\infty(M_1\setminus \pl M)$.
Using Hodge decomposition, let us write $f'=dp+q$ for some $p\in H^{-k+1}(M_1^\circ)$ and $q\in H^{-k}(M_1^\circ)$ that are also in $C^\infty(M_1\setminus \pl M)$ and with $D^*q=0$ (recall $D=d$ on functions). We have $\Delta_gp=D^*dp=D^* f'$ thus $p\in C^\infty(M_1\setminus \pl M)$ and $q\in C^\infty(M_1\setminus \pl M)$. Since $P_1dw=0=P_2dw$ for all $w\in C^\infty(M_1\setminus \pl M)$ such that $w|_{M_1^\circ}\in H^{-k}(M_1^\circ)$ and $w|_{\pl M_1}=0$, 
we have $Pdp=Pd\til{p}$ where $\til{p}\in C^\infty(M_1)$ is any function such that $\til{p}=p$ near $\pl M_1$. Thus $Pq=Pf'-Pdp=-Pd\til{p}\in C^\infty(M_1^\circ)$ and thus $q\in C^\infty(M_1)$ by ellipticity of $P$ near $\pl M$ since $D^*q=0$. Then on $M_1^\circ$, 
$P(d\til{p}+q)=0$ and by Theorem \ref{th4} we deduce that 
$d\til{p}+q=d\hat{p}$ for some $\hat{p}\in C^\infty(M_1)$ with $\hat{p}|_{\pl M_1}=0$. Thus $f'=d(p+\hat{p}-\til{p})$ 
and we have 
\[ 0=\cjg Ef,f'\cjd_{L^2(M_1)}=\cjg (1+\Delta_g)^{k/2}f,(1+\Delta_g)^{k/2}f\cjd_{L^2(M)}
\]
which shows that $f=0$. We conclude that for each $f\in H^k(M)\cap \ker D^*$ there exists $u\in H^{k-1}(M_1)$ such that $r_MPu=f$, in other words ${\pi_1}_*w=f$ in $M$ if 
\[w:= (I^{M_1})^*\chi_0^2I_1^{M_1}u+\sum_{j=1}^J(I^{\chi_j})^*I^{\chi_j}_1u.\]
Using \eqref{regI*} we get $w\in H^{k-1}(SM_1)$ and $Xw=0$, ending the proof.
\end{proof}
\begin{rem}\label{lowregI1}
For $f\in L^2(M)$, $I_1f=0$ implies $Pf=0$ where $P$ is the operator on $M_1$ defined in the proof of Proposition \ref{surjI_1*} and $f$ is extended by $0$ in $M_1\setminus M$. We can use Hodge decomposition $f=q+dp$ for 
$p\in H_0^1(M_1)\cap C^\infty(M_1\setminus M)$ with $q\in L^2(M_1)\cap C^\infty(M_1\setminus M)$ solving $D^*q=0$: then $Pq=0$, and by ellipticity of $P$ on $\ker D^*$, we obtain $q\in C^\infty(M_1)$; then applying 
Theorem \ref{th4} on $M_1$ we get $q=0$, thus $f=dp$ on $M$ with $p\in H_0^1(M)$. 
\end{rem}

\section{Scattering rigidity and injectivity of $X$-ray on tensors in $2d$}

\subsection{Scattering rigidity: proof of Theorem \ref{th2}}

Let $(M,g)$ be a compact Riemannian surface with boundary. As before, we denote by $X$ the geodesic vector field associated to $g$, and by $V$ the vertical vector field defined by $Vf:=\pl_\theta(R^*_\theta f)|_{\theta=0}$ for each $f\in C^\infty(SM)$, where $R_\theta:SM\to SM$ is the rotation of angle $+\theta$ in the fibers of $SM$. We also set $X_\perp:=[X,V]$. The vector fields $X,V,X_\perp$ span the tangent bundle $T(SM)$, and are an orthonormal basis for the Sasaki metric on $SM$. 
Since $SM$ is a circle bundle, every $f\in C^\infty(SM)$ admits a Fourier decomposition in the fibers
\begin{equation}\label{decompHk}
 f(x,v)=\sum_{k\in \zz}f_k(x,v),\qquad\mbox{where }Vf_k=ikf_k.
\end{equation}
The above series converges uniformly, and $f_k\in C^\infty(SM)$. The analogous Fourier decomposition also holds for $f\in L^2(SM)$, where now $f_k\in L^2(SM)$ and the sum weakly converges.  We refer the reader to \cite{Guillemin:1980rq,Paternain:2013fx} for more details.

Consider the fiberwise Hilbert transform $H:C^\infty(SM)\to C^\infty(SM)$, given by
\begin{align*}
 H(w):=-\sum_{k\in\Z} i\, {\rm sign}(k)w_k,\qquad\forall w=\sum_{k\in\Z}w_k\in C^\infty(SM),
\end{align*}
with the convention that ${\rm sign}(0)=0$. The operator $H$ extends continuously to $L^2(SM)$ and we decompose it as $H=H_{\rm ev}+H_{\rm od}$, where
\[
 H_{\rm ev}(w):=-\sum_{k\ \mathrm{even}} i\, {\rm sign}(k)w_k,\qquad
  H_{\rm od}(w):=-\sum_{k\ \mathrm{odd}} i\, {\rm sign}(k)w_k.
\]
For each $f\in C^\infty(\pl M)$, we denote by $\mc{P}(f)$ its harmonic extension to $(M,g)$. From now on, let us assume that $(M,g)$ is non-trapping, so that the forward exit time function $\tau_{g}^+:\partial SM\to[0,\infty)$ is everywhere finite, and the \emph{scattering map}
\begin{equation}\label{scatmap}
\sigma_g: \pl_-SM\cup \pl_0SM\to \pl_+SM\cup \pl_0SM, 
\qquad 
\sigma_g(y):=\varphi_{\tau^+_{SM}(y)}(y)
\end{equation} 
is well defined.

\begin{prop}\label{reducPU}
Let $f^*\in C^\infty(\pl M;\R)$ and $w\in W^{1,\infty}(SM;\R)$ such that $Xw=0$. Then
\begin{equation} \label{identite1}
(\sigma_{g}^*-{\rm Id})(H_{{\rm ev}}w)|_{\pl_-SM\cup \pl_0SM}=(\sigma_{g}^*-{\rm Id})(\pi_0^*f^*)|_{\pl_-SM\cup \pl_0SM}
\end{equation}
if and only if  $w_0|_M-i\mc{P}(f^*)$ is holomorphic.
\end{prop}

\begin{proof}
We first recall the Pestov-Uhlmann relation \cite[Theorem~1.5]{Pestov:2005jo}:
\begin{equation}\label{peuh}
H_{\rm od}Xw-XH_{\rm ev}w=X_\perp w_0,
\qquad
\forall w \in W^{1,\infty}(SM).
\end{equation}
Assume that  $w_0|_M-i\mc{P}(f^*)$ is holomorphic in $M$. Notice that for almost all $(x,v)\in SM$
\begin{align*}
X_\perp w_0(x,v)=
*dw_0(x)v= -d \mc{P}(f^*)(x)v
= -X\mc{P}(f^*)(x,v).
\end{align*}
where $*$ denotes the Hodge star.
Moreover, $IXu=(\sigma_{g}^*-{\rm Id})(u|_{\pl SM})$ for each $u\in 
W^{1,\infty} (SM)$. Therefore, by applying $I$ to both sides of equation~\eqref{peuh}, we obtain
\[ 
(\sigma_{g}^*-{\rm Id})(H_{\rm ev}w)|_{\pl_-SM\cup \pl_0SM}=IX\mc{P}(f^*)= (\sigma_{g}^*-{\rm Id})(\pi_0^*f^*)|_{\pl_-SM\cup \pl_0SM}.
\] 
Conversely, assume that~\eqref{identite1} holds. Let $q\in C^\infty(M;\R)$ such that $q|_{\pl M}=f^*$. Then
\begin{align*}
IXH_{\rm ev}w=
(\sigma_{g}^*-{\rm Id})(H_{{\rm ev}}w)|_{\pl_-SM\cup \pl_0SM}=
(\sigma_{g}^*-{\rm Id})(\pi_0^*f^*)|_{\pl_-SM\cup \pl_0SM}= I_1(dq).
\end{align*}
By the Pestov-Uhlmann relation~\eqref{peuh}, we have for almost all $(x,v)\in SM$
\begin{align*}
XH_{\rm ev}w(x,v)=-X_\perp w_0(x,v)=-*dw_0(x)v.
\end{align*}
Therefore $*dw_0+dq\in\ker I_1\cap L^2(M;T^*M)$. By Remark~\ref{lowregI1}, there exists $p\in H_0^1(M)$ real-valued such that $*dw_0=d(p-q)$. This implies that $w_0|_M+i(p-q)$ is holomorphic in $M$, and since $(q-p)|_{\pl M}=f^*$ we have $q-p=\mc{P}(f^*)$.
\end{proof}

We are now in position to prove Theorem \ref{th2}.

\begin{proof}[Proof of Theorem \ref{th2}.]
Since $g_1$ and $g_2$ agree to order $2$ at the boundary, we can identify a neighborhood $U_1\subset M_1$ of $\partial M_1$ with a neighborhood $U_2\subset M_2$ of $\partial M_2$ by means of a diffeomorphism, in such a way that the 2-jets of $g_1$ and $g_2$ agree on $\partial M_1\equiv\partial M_2$. Let $C$ be a collar extension for $M_1$. By means of the above identification, $C$ is a collar for $M_2$ as well. We extend $g_1$ to a smooth Riemannian metric on $N_1:=M_1\cup C$. We also extend $g_2$ to $N_2:=M_2\cup C$ by setting $g_2:=g_1$ on $C$. Notice that $g_1$ is a $C^1$ Riemannian metric, piecewise smooth on $N_2$. We denote by $X_i$ and $\varphi^{g_i}_t$ the geodesic vector field and the geodesic flow on the unit tangent bundle $SN_i$, and by $\mc{P}_i(f)$ the harmonic extension (with respect to the metric $g_i$) of a function $f\in C^\infty(\pl M_i)$ to $M_i$. Notice that $X_1$ is smooth, whereas $X_2$ is only Lipschitz and piecewise smooth with singularities contained in $\pl SM_2$. In particular, the geodesic flow $\varphi_t^{g_i}$ is Lipschitz. Up to shrinking the collar $C$, Lemma~\ref{l:no_conjugate_points} guarantees that $(N_i,g_i)$ is non-trapping.

Consider two functions $f,f^*\in C^\infty(\partial M_1)$ such that $\mc{P}_1(f)-i\mc{P}_1(f^*)$ is holomorphic on $(M_1,g_1)$. We wish to prove that $\mc{P}_2(f)-i\mc{P}_2(f^*)$ is holomorphic on $(M_2,g_2)$. This will mean that the set of restrictions at the boundary of holomorphic functions are the same for $(M_1,g_1)$ and $(M_2,g_2)$, and \cite[Theorem~1]{Belishev:2003ss} will imply the existence of a diffeomorphism $\psi:M_1\to M_2$ that extends the identification of $\partial M_1$ with $\partial M_2$ and satisfies $\psi^*g_2=e^\rho g_1$ for some $\rho\in C^\infty(M_1)$ with $\rho|_{\partial M_1}\equiv0$.

By Proposition \ref{Pinjective}, there exists $w'\in C^\infty(SN_1^\circ )$ so that $X_1w'=0$ and $w'_0=\mc{P}_1(f)$ in $M_1$. Proposition~\ref{reducPU} implies that
\begin{equation*} \label{identity_w'}
(\sigma_{g_1}^*-{\rm Id})(H_{{\rm ev}}w')|_{\pl_-SM_1\cup \pl_0SM_1}=(\sigma_{g_1}^*-{\rm Id})(\pi_0^*f^*)|_{\pl_-SM_1\cup \pl_0SM_1}
\end{equation*}
Since $X_1|_C=X_2|_C$,  $\sigma_{g_1}=\sigma_{g_2}$, and $(M_2,g_2)$ is non-trapping,  there exists a unique function $w'':SN_2^\circ\to\R$ such that $w''|_C=w'|_C$ and $X_2w''\equiv0$: it is simply given by $w''=w'$ on $C$ and $w''(y)=w'(\tau^+_{g_2}(y))$ for $y\in SM_2$. 
We claim that $w''$ is Lipschitz continuous. Indeed, it is smooth in $C$. Moreover, since $(M_2,g_2)$ is non-trapping, for each $y\in SM_2$ there exists $t\in\R$ such that $\varphi_t^{g_2}(y)\in C^\circ$. If $W\subset SN_2$ is a sufficiently small open neighborhood of $y$, we have that $\varphi_t^{g_2}(W)\subset C$. Therefore 
\[w''|_W=w''\circ\varphi^{g_2}_{-t}|_{C}\circ\varphi_t^{g_2}|_W = w''|_C\circ\varphi_t^{g_2}|_W =w'|_C\circ\varphi_t^{g_2}|_W, \]
which implies that $w''|_W\in W^{1,\infty}(W)$. 
Since $\sigma_{g_1}=\sigma_{g_2}$ and $w'|_{\partial SM_1}=w''|_{\partial SM_2}$, equation~\eqref{identity_w'} can be rewritten as
\begin{align*}
(\sigma_{g_2}^*-{\rm Id})(H_{{\rm ev}}w'')|_{\pl_-SM_2\cup \pl_0SM_2}=(\sigma_{g_2}^*-{\rm Id})(\pi_0^*f^*)|_{\pl_-SM_2\cup \pl_0SM_2}.
\end{align*}
Therefore, Proposition~\ref{reducPU} implies that $w''_0|_{M_2}-i\mc{P}_2(f^*)$ is holomorphic on $(M_2,g_2)$. Finally, since $w''_0|_{\partial M_2}=w'_0|_{\partial M_1}=f$, we have $w''_0|_{M_2}=\mc{P}_2(f)$.
\end{proof}

\subsection{Injectivity of the X-ray transform on tensors on surfaces}

Consider the fiberwise Fourier decomposition of some $f\in L^2(SM)$. By means of the pull-backs $\pi_{|k|}^*$ of the maps of Equation~\eqref{e:pi_m}, each summand $f_k$ can be viewed as a section of the line bundle $\mc{L}^k:=\otimes^k(T^*M)^{1,0}$ if $k>0$, or of the line bundle $\mc{L}^k:=\otimes^k(T^*M)^{0,1}$ if $k<0$. We set $H_k:=\pi_{|k|}^*(L^2(M;\mc{L}^k))$, so that $L^2(SM)=\oplus_{k\in\Z} H_k$. The operators 
\[\eta_\pm:= \demi(X\pm iX_\perp) : H_{k}\cap C^\infty(SM)\to H_{k\pm 1}\cap C^\infty(SM) \]
satisfy $\eta_+^*=-\eta_-$ on $H_0^1(SM)\cap H_k$.  If $k\geq0$, any $u\in H_k$ satisfying $\eta_-u=0$ is a holomorphic section of $\mc{L}^k$, whereas any $w\in H_{-k}$ satisfying $\eta_+w=0$ is a antiholomorphic section of $\mc{L}^{-k}$.

Using Proposition \ref{surjI_1*}, we show the following result that implies Theorem \ref{th3}. 
\begin{thm}\label{injImsurf}
Let $(M,g)$ be a compact Riemannian surface with boundary that is non-trapping and without conjugate points. If $m\geq0$ and $f\in H^1(M; \otimes_S^mT^*M)\cap\ker(I_m)$, there exists $p\in H^2(M; \otimes_S^{m-1}T^*M)$ such that $p|_{\pl M}\equiv0$ and $f=Dp$.
\end{thm}
\begin{proof} We use an argument somehow similar to \cite[Theorem 9.3]{Paternain:2014lr} based on an induction on $m$. For notational simplicity, we shall identify $H_j$ with $L^2(M;\mc{L}^j)$ through $\pi_j^*$.
Assume that we have proved that for $m\geq 2$, if $f\in H^1(M;\otimes_S^{m-1} T^*M)$ and $Xu=\pi_{m-1}^*f$ for some $u\in H^1(SM)$ with $u|_{\pl SM}=0$, then $u\in \pi_{m-2}^*(H^1(M;\otimes_S^{m-2}T^*M))$.
Let now $f\in H^1(M;\otimes_S^m T^*M)$ be real valued and decompose it as $\pi_m^*f=\sum_{k=0}^m f_{m-2k}$ with 
$f_j\in H_j$ using  \eqref{decompHk}; 
note that $\bbar{f_j}=f_{-j}$. Using the orthogonal decomposition 
$H^1(M;\mc{L}^m)=\ker \eta_-\oplus {\rm Im}\, \eta_+$ where $\eta_+$ acts on 
$H^2(M;\mc{L}^{m-1})\cap H_0^1(M;\mc{L}^{m-1})$, we can write $f_m=\eta_+h_{m-1}+q_m$
with $\eta_-q_m=0$ and $h_{m-1}\in H^2(M;\mc{L}^{m-1})$ such that $h_{m-1}|_{\pl M}=0$. Since $I_mf=0$, Proposition \ref{Livsic} implies that there is $u\in H^1(SM)$ such that $u|_{\pl SM}=0$ and $Xu=\pi_m^*f$. 
We have $X(u-h_{m-1})=q_m-\eta_-(h_{m-1})+\sum_{k=1}^{m}f_{m-2k}$.
Using that $\mc{L}$ is holomorphically trivial (since $M$ has boundary), there is 
$q_1\in C^\infty(M;\mc{L})\cap \ker \eta_-$ such that $q_m=aq_1^{m}$ for some holomorphic 
$a\in H^1(M)\cap \ker \eta_-$. By Proposition \ref{surjI_1*}, there is $w\in H^k(SM)$ 
for $k\in \nn$ as large as we like such that ${\pi_1}_*w=(2\pi)^{-1}q_1$ and $Xw=0$ on $SM$ (we used $D^*q_1=\eta_-q_1=0$), and there is $\alpha\in L^2(SM)$ such that $X\alpha=0$ in $H^{-1}(SM)$ and ${\pi_0}_*\alpha=(2\pi)^{-1}a$.
Note that we can assume that $w(x,-v)=-w(x,v)$ and $\alpha(x,-v)=\alpha(x,v)$
since $X$ maps even functions with respect to the involution $(x,v)\mapsto (x,-v)$ to odd functions and conversely.
 For each $u\in L^2(SM)$, we shall denote $u_>:=\sum_{j\geq 0}u_j$ if $u=\sum_{j\in\zz} u_j$ is the Fourier decomposition in fibers of $u$.
It is then direct to check that $X(w_>)=0$ and $X(\alpha_>)=0$ by using that 
$\eta_-q_1=0$, $\eta_-a=0$. Consequently, if $\hat{w}:=\alpha_>w_{>}^m\in L^2(SM)$, we get $X\hat{w}=0$ in $H^{-1}(SM)$  and the Fourier decomposition of $\hat{w}$ is of the form
$\hat{w}=\sum_{k\geq m}\hat{w}_k$ with  $\hat{w}_m=aq_1^m$. 
Then 
\[ \|q_m\|_{L^2}^2=\cjg \hat{w}_m,q_m\cjd_{L^2}
=\cjg \hat{w},X(u-h_{m-1})\cjd_{L^2}=\cjg X\hat{w},h_{m-1}-u\cjd_{L^2}=0.\]
We also have $f_{-m}=\bbar{f_m}=\eta_-\bbar{h_{m-1}}$ thus 
$X(u-h_{m-1}-\bbar{h_{m-1}})\in \oplus_{|j|\leq m-2} H_{j}$ with $H^1(SM)$ regularity. Then by using the induction assumption, we have that $u-h_{m-1}-\bbar{h_{m-1}}\in \sum_{|j|\leq m-2}H_j$, which proves the induction property at step $m$. Since the induction property is also true for $m=2$ by Theorem \ref{th4} (more precisely by Remark \ref{lowregI1}), 
the proof is done. The $H^2$ regularity of $p$  comes from ellipticity of $D$.
\end{proof}

\section{Lens and boundary rigidity}

\subsection{Geodesics on compact surfaces without conjugate points}

Before proving Theorem~\ref{th1}, we need a few preliminary statements concerning geodesics on surfaces. 
The first lemma is a geodesic analogue of the classical Poincar\'e-Bendixson's Theorem, and we leave its proof as an exercise. As usual, by simple closed geodesic we mean a geodesic that is the image of a smooth embedding of a circle into the Riemannian manifold. 

\begin{lem}\label{l:omega_limit}
Let $(M,g)$ be a compact, simply connected, Riemannian surface (with or without boundary). If $\gamma:[0,\infty)\hookrightarrow M$ is a geodesic parametrized with constant speed and without self-intersections, then the $\omega$-limit set of $\gamma$ contains a simple closed geodesic.
\hfill\qed 
\end{lem}

Let $(M,g)$ be a Riemannian manifold. For each $x,y\in M$, we set
\begin{align*}
\Omega_{xy}M
 :=
\big\{
\gamma\in H^{1}([0,1];M)\ |\ \gamma(0)=x,\ \gamma(1)=y
\},\qquad 
\Omega_xM  :=\Omega_{xx}M.
\end{align*}
We consider the energy and length functions on $\Omega_{xy}M$, which are given by
\begin{align*}
\energy(\gamma) = \int_0^1 \|\dot\gamma(t)\|^2_g\,d t,\qquad 
\length(\gamma) = \int_0^1 \|\dot\gamma(t)\|_g\,d t.
\end{align*}
Notice that $\length(\gamma)^2\leq\energy(\gamma)$, with equality if and only if $\gamma$ has constant speed. These two functions have the same critical points, which are precisely the geodesics parametrized with constant speed and, when $x=y$, the constant curves at $x$. Moreover, a geodesic $\gamma\in\Omega_{xy}M$ is a local minimum of one of these functions if and only if it is a local minimum of the other one. The only reason for introducing the energy $E$ here is that it has better functional properties: smooth regularity and the Palais-Smale condition, see e.g.~\cite{Klingenberg:1978so}. We will denote the sublevel sets of the energy function by
\begin{align*}
 \Omega_{xy}^{\leq\ell}M
& :=
\big\{
\gamma\in\Omega_{xy}M\ \big|\ \energy(\gamma)\leq \ell^2
\big\},\qquad \ell>0.
\end{align*}

\begin{lem}\label{l:no_self_intersections}
Any simply connected, Riemannian surface without conjugate points does not have geodesics with self-intersections.
\end{lem}

\begin{proof}
Assume by contradiction that one such Riemannian surface $(M,g)$ contains a geodesic $\gamma:[0,1]\to M$ such that $\gamma(0)=\gamma(1)=:x$. Since $M$ is simply connected, $\gamma$ bounds a compact disk $D\subset M$.
We have two cases to consider. The first one is when $\dot\gamma(1)$ does not point inside $D$, meaning that it either points outside $D$ or $\dot\gamma(1)=\dot\gamma(0)$. In this case, $D$ is convex (although not strictly), and a well known shortening procedure due to Birkhoff provides continuous maps $B=B_\ell:\Omega_x^{\leq\ell}D\to\Omega_x^{\leq\ell}D$, for $\ell>0$, such that $\energy(B(\zeta))\leq\energy(\zeta)$, the fixed points of $B$ are precisely the geodesic loops in $\Omega_x^{\leq\ell}D$, and every sequence $\{B^n(\zeta)\ |\ n\in\N\}$ converges to a geodesic loop or to a constant as $n\to\infty$, see \cite[Prop.~A.1.2]{Klingenberg:1978so}. Since $(M,g)$ has no conjugate points, the Morse Index Theorem from Riemannian geometry implies that the geodesic loop $\gamma$ is a local minimizer of the energy function, that is, for any loop $\zeta\in \Omega_xD\setminus\{\gamma\}$ that is sufficiently close to $\gamma$ we have $\energy(\gamma)<\energy(\zeta)$. Actually, for any sufficiently small open neighborhood $U\subset\Omega_xM$ of $\gamma$, there exists $\epsilon>0$ such that $\energy|_{\partial U}\geq\energy(\gamma)+\epsilon$. We fix one such open set $U$ that is small enough so that it does not contain the stationary curve at $x$, and we consider the corresponding $\epsilon>0$. Let $h:[0,1]\to\Omega_xD$ be a homotopy such that $h(0)=\gamma$ and  $h(1)$ is the stationary point at $x$, and fix $\ell>0$ large enough so that $h([0,1])\subset\Omega_x^{\leq\ell}D$. Notice that $B^n(h(0))=\gamma$ and $B^n(h(1))=h(1)$ for all $n\in\N$. We set $s_n\in[0,1]$ such that 
\[\energy(B^n(h(s_n)))=\max\big\{\energy(B^n(h(s)))\ \big|\ s\in[0,1] \big\}\]
Notice that $\energy(B^n(h(s_n)))\geq \energy(\gamma)+\epsilon$. By a standard compactness argument, a subsequence of $\{B^n(h(s_n))\ |\ n\in\N\}$ converges to a geodesic loop $\zeta\in\Omega_xM$ that is not a local minimum of the energy function, but then the Morse Index Theorem implies that $\zeta$ contains a pair of conjugate points. This gives a contradiction.

It remains to consider the case when $\dot\gamma(1)$ points inside $D$. The argument of the previous paragraph cannot be directly applied, for the Birkhoff shortening map $B$ does not preserve the disk $D$ anymore. Instead, let us extend the geodesic $\gamma$ to its maximal non-negative interval of definition $I\subset[0,\infty)$. If the extended $\gamma:I\to M$ does not have other self-intersections other than at time $0$ and $1$, then $I=[0,\infty)$ and the curve $\gamma$ is entirely contained in $D$; by Lemma~\ref{l:omega_limit}, the $\omega$-limit set of $\gamma$ must contain a simple closed geodesic $\gamma'$; by applying the argument of the previous paragraph to $\gamma'$ we thus obtain a contradiction. Therefore, $\gamma:I\to M$ must have other self intersections. In particular there exists a compact sub-interval $[a,b]\subset I$ such that $\gamma|_{(a,b)}$ is without self intersections, $\gamma(a)=\gamma(b)$, and $\dot\gamma(b)$ does not point inside the disk $D'$ bordered by $\gamma|_{[a,b]}$. The argument of the previous paragraph applied to $\gamma|_{[a,b]}$ provides a contradiction.
\end{proof}

\begin{lem}\label{l:no_intersections}
In a simply connected, Riemannian surface without conjugate points, no pair of distinct points can be joined by more than one geodesic.
\end{lem}

\begin{proof}
Let $(M,g)$ be a compact, simply connected, Riemannian surface without conjugate points. If $\gamma_0:[0,1]\to M$ and $\gamma_1:[0,1]\to M$ are distinct geodesics with the same endpoints $x:=\gamma_0(0)=\gamma_1(0)$ and $y:=\gamma_0(1)=\gamma_1(1)$, we have that $\dot\gamma_0(0)$ and $\dot\gamma_1(0)$ are linearly independent, and $\gamma_0(t)\neq\gamma_1(s)$ for all $t>0$ and $s>0$ sufficiently small. Therefore, up to replacing $\gamma_0$ and $\gamma_1$ with two suitable subcurves, we can assume without loss of generality that $\gamma_0$ and $\gamma_1$ only intersects at their endpoints $x$ and $y$. The closed curve $\gamma_0*\overline\gamma_1$ bounds a disk $D\subset M$.

There are two cases to consider. The first, easier, one is when the disk $D$ is (not strictly) convex, that is, the unsigned angles formed by $\gamma_0$ and $\gamma_1$ measured from inside $D$ are strictly less $\pi$. In this case, we reach a contradiction by means of a minmax argument as in the first paragraph of the proof of Lemma~\ref{l:no_self_intersections}.

The remaining case to consider is the one in which  $D$ is not convex. This means that at least one vector between $\dot\gamma_0(1)$ and $-\dot\gamma_0(0)$ points inside $D$. Let us extend the geodesic $\gamma_0$ to its maximal interval of definition $I\subset\R$. We set $b_0:=\sup\{t\in I\ |\ \gamma_0([1,t])\subset D\}$. Lemma~\ref{l:no_self_intersections} implies that $\gamma_0:I\to M$ has no self-intersections. The time $b_0$ is finite, for otherwise, by Lemma~\ref{l:omega_limit}, the $\omega$-limit set of $\gamma_0$ would contain a simple closed geodesic in $D$, which again is prevented by Lemma~\ref{l:no_self_intersections}. Therefore, there exists $b_1\in(0,1]$ such that $\gamma_0(b_0)=\gamma_1(b_1)$. Analogously, the time $a_0:=\inf\{t\in I\ |\ \gamma_0([t,0])\subset D\}$ is finite, and there exists $a_1\in[0,b_1)$ such that $\gamma_0(a_0)=\gamma_1(a_1)$. Summing up, the geodesics $\gamma_0|_{[a_0,b_0]}$ and  $\gamma_1|_{[a_1,b_1]}$ join the same endpoints, do not intersect elsewhere, and the loop $\gamma_0|_{[a_0,b_0]}*\overline{\gamma_1|_{[a_1,b_1]}}$ bounds a convex disk $D'$. We thus repeat the arguments of the previous paragraph for $D'$ instead of $D$, and obtain a contradiction.
\end{proof}

\begin{lem}\label{l:scs_non_trapping}
Any compact, simply connected, Riemannian surface with no conjugate points is non-trapping.
\end{lem}

\begin{proof}
Let $(M,g)$ be a compact, simply connected, Riemannian surface, and assume by contradiction that there exists a geodesic $\gamma:[0,\infty)\to M$ parametrized with constant speed. By Lemma~\ref{l:omega_limit}, its $\omega$-limit set contains a closed geodesic. However, a closed geodesic is in particular a geodesic with a self-intersection, and Lemma~\ref{l:no_self_intersections} provides a contradiction. 
\end{proof}

Given a Riemannian manifold $(M,g)$, we denote by $d:M\times M\to[0,\infty)$ its induced Riemannian distance. The following statement is due to Croke and Wen \cite[Corollary~2]{Croke:2015qy}, and we provide here an alternative proof.

\begin{prop}\label{p:unique_minimizer}
Let $(M,g)$ be a compact, simply connected, Riemannian surface with non-empty boundary and no conjugate points. For each pair of distinct points $x,y\in M$ there is at most one geodesic $\gamma:[0,1]\to M$ such that $\gamma(0)=x$ and $\gamma(1)=y$. If such a geodesic exists, it is the only curve satisfying $\length(\gamma)=d(x,y)$.
\end{prop}

\begin{proof}
By Lemma~\ref{l:scs_non_trapping}, $(M,g)$ is non-trapping. We can thus apply  Lemma~\ref{l:no_conjugate_points}, and embed $(M,g)$ in the interior of a larger compact Riemannian surface with boundary $(M_1,g)$ without conjugate points. We fix $\epsilon>0$ smaller than the injectivity radius of $(M_1,g)$ and such that $M_1$ contains a $2\epsilon$-neighborhood of $M$ in its interior. We stress that the Riemannian distance $d:M\times M\to[0,\infty)$ is not the restriction of the Riemannian distance of $(M_1,g)$.
We prove the proposition by contradiction, assuming that there exists a geodesic $\gamma$ as in the statement, and a different curve $\zeta:[0,1]\to M$ parametrized with constant speed such that $\zeta(0)=\gamma(0)=x$, $\zeta(1)=\gamma(1)=y$, and $d(x,y)=\length(\zeta)\leq\length(\gamma)$. Notice that $\length(\zeta|_{[t_0,t_1]})=d(\zeta(t_0),\zeta(t_1))$ for all $[t_0,t_1]\subset[0,1]$.

Up to replacing $\gamma$ and $\zeta$ with subcurves, we can assume that they only intersect at their endpoints, so that their concatenation $\gamma*\overline\zeta$ borders a compact disk $D\subseteq M$. We consider the compact subset $T:=\{t\in[0,1]\ |\ \zeta(t)\in\partial M\}$. Since $\zeta$ is a global minimizer of the length function on the path space $\Omega_{xy}M$, it is $H^2$, and actually twice differentiable outside a countable set $K\subset T$; moreover, 
\begin{align}
\label{e:concavity_minimizer_1}
\nabla_t\dot\zeta(t)=0, & \qquad \forall t\in [0,1]\setminus T,\\
\label{e:concavity_minimizer_2}
g(\nabla_t\dot\zeta(t),\nu(\zeta(t)))\leq0, & \qquad \forall t\in T\setminus K,
\end{align}
where $\nu$ is the inward pointing normal vector field to $\partial M$, see \cite{Alexander:1987rc}. Notice that $\length(\zeta)$ must be larger than the injectivity radius of $(M_1,g)$ (and thus larger than $\epsilon$), for otherwise the geodesic $\gamma$ would be shorter than $\zeta$. Let $\tau>0$ be such that $\length(\zeta|_{[0,\tau]})=\epsilon$. We denote by $\exp$ the exponential map of $(M_1,g)$, which is an extension of the exponential map of $(M,g)$. The geodesic \[\gamma_1(t):=\exp_x(t\tau^{-1}\exp_x^{-1}(\zeta(\tau)))\in M_1\] is well defined for $t\in[0,\tau]$; indeed, $\length(\gamma_1|_{[0,\tau]})\leq\epsilon$ and therefore $\gamma_1|_{[0,\tau]}$ cannot reach the boundary of $M_1$. Let $I_1\subset[0,\infty)$ be the maximal non-negative interval of definition of $\gamma_1$. We have two possible cases to consider:
\begin{itemize}
\item The curve $\zeta|_{[0,\tau]}$ is a geodesic of $(M,g)$, and therefore $\zeta|_{[0,\tau]}=\gamma_1|_{[0,\tau]}$. In this case, $\zeta$ cannot entirely coincide with $\gamma_1|_{[0,1]}$, for otherwise we would have two distinct geodesics $\gamma$ and $\gamma_1$ intersecting at their endpoints, which would violate Lemma~\ref{l:no_intersections}. Let $s_1:=\max\big\{t\geq\tau\ \big|\ \gamma_1|_{[0,t]}=\zeta|_{[0,t]}\big\}<1$. By the non-convexity~\eqref{e:concavity_minimizer_2}, we have that $\gamma_1(s_1+t)\in D^\circ$ for all $t>0$ small enough. 

\item The curve $\zeta|_{[0,\tau]}$ is not a geodesic. Therefore, by the non-convexity assumption~\eqref{e:concavity_minimizer_2}, the velocity vectors $\dot\zeta(\tau)$ and $\dot\gamma_1(\tau)$ are linearly independent, and $\dot\gamma_1(\tau)$ points in the interior of $D$. We set $s_1:=\tau$, and notice that $\gamma_1(s_1+t)\in D^\circ$ for all $t>0$ small enough.
\end{itemize}
Since $(M,g)$ is non-trapping, the geodesic $\gamma_1$ must eventually exit $D$ in forward time. Lemma~\ref{l:no_intersections} prevents $\gamma_1$ to intersect $\gamma$. Therefore, $\gamma_1|_{I_1\setminus[0,s_1]}$ must intersect $\zeta$, and we set
$s_1':=\min\{t>s_1\ |\ \gamma_1(t)\in\partial D\}$, $s_1'':=\zeta^{-1}(\gamma_1(s_1'))$, and $\zeta_1:=\zeta|_{[s_1,s_1'']}$. We now repeat the whole argument replacing the curves $\gamma$ and $\zeta$ with $\gamma_1:[s_1,s_1']\to M$ and $\zeta_1:[s_1,s_1'']\to M$ respectively. Notice that $\length(\gamma_1|_{[s_1,s_1']})\geq\length(\zeta_1)=d(\zeta_1(s_1),\zeta_1(s_1''))$ and $\length(\zeta_1)\leq\length(\zeta)-\epsilon$.

At every iteration of this inductive procedure, we shrink the length of the minimizing curve by at least $\epsilon$. Therefore, after at most $n:=\lfloor d(x,y)\epsilon^{-1} \rfloor$ iterations, we are left with a geodesic $\gamma_n:[s_n,s_n']\to M$ and a different curve $\zeta_n:[s_n,s_n'']\to M$ such that $\gamma_n(s_n)=\zeta_n(s_n)=:x_n$, $\gamma_n(s_n')=\zeta_n(s_n'')=:y_n$, $\length(\gamma_n)\geq\length(\zeta_n)=d(x_n,y_n)$, and $0<\length(\zeta_n)\leq\epsilon$. Lemma~\ref{l:no_intersections} implies that $\gamma_n$ is the only geodesic of $(N,g)$ joining $x_n$ and $y_n$. 
Since $\epsilon$ is smaller than the injectivity radius of $(N,g)$, we must have
$d(x_n,y_n)=\length(\gamma_n)$, and since $\zeta_n$ is not the same curve as $\gamma_n$, we must have $\length(\zeta_n)>\length(\gamma_n)$, which contradicts the fact that $\length(\zeta_n)=d(x_n,y_n)$.
\end{proof}

\subsection{Relations between boundary distance and lens data}\label{relationslensbeta}
Let $(M,g)$ be a compact Riemannian manifold with boundary, whose geodesic flow is non-trapping. As before, we can see $(M,g)$ as a compact subset of an auxiliary closed Riemannian manifold $(N,g)$, so that the geodesic flow $\varphi_t:SN\to SN$ is complete. We consider the forward hitting time function that we introduced in~\eqref{hitting}, namely the function
\begin{align}\label{htf} 
 t_g^+:\pl_- SM\cup\pl_0SM\to[0,\infty),
 \qquad
 t_g^+(y):=\inf\{ t>0\ |\ \varphi_t(y)\not\in  SM^\circ\}.
\end{align}
We define the \emph{hitting scattering map}
\begin{align}\label{hsm}
s_g:\pl_- SM\cup\pl_0SM\to\pl_+ SM\cup\pl_0SM,
\qquad
s_g(y):=\varphi_{t^+_g(y)}(y).
\end{align}
It turns out that the hitting lens data $(t_g^+,s_g)$ provides the same information as the usual exit lens data $(\tau_g^+,\sigma_g)$, in the following sense.

\begin{prop}\label{equivalencescat}
Let $M$ be a connected compact manifold with boundary, and $g_1,g_2$ two Riemannian metrics on $M$ that are non-trapping and $g_1|_{T\pl M}=g_2|_{T\pl M}$. Then $(t^+_{g_1},s_{g_1})=(t^+_{g_2},s_{g_2})$ if and only if $(\tau^+_{g_1},\sigma_{g_1})=(\tau^+_{g_2},\sigma_{g_2})$.
\end{prop}
\begin{proof} 
Assume that $(\tau^+_{g_1},\sigma_{g_1})=(\tau^+_{g_2},\sigma_{g_2})=:(\tau^+,\sigma)$, and consider an arbitrary $y\in \pl_-SM\cup \pl_0SM$. If $\tau^+(y)=0$, we have $s_{g_1}(y)=s_{g_2}(y)=y$ and $t^+_{g_1}(y)=t^+_{g_2}(y)=0$. Otherwise, if $\tau^+(y)>0$, there exists a unique point $z\in \sigma^{-1}(\sigma(y))$ such that 
\[\tau^+(z)=\sup\big\{\tau^+(z')\ \big|\  z'\in \sigma^{-1}(\sigma(y)),\ \tau^+(z')<\tau^+(y)\big\},\] 
and we conclude that $s_{g_1}(y)=s_{g_2}(y)=z$ and $t^+_{g_1}(y)=t^+_{g_2}(y)=\tau^+(y)-\tau^+(z)$.

Conversely, assume that $(t^+_{g_1},s_{g_1})=(t^+_{g_2},s_{g_2})$. By Stefanov-Uhlmann \cite[Theorem 1]{Stefanov:2009lp}, up to replacing $g_2$ with $\psi^*g_2$ for a suitable diffeomorphism $\psi:M\to M$ with $\psi|_{\pl M}={\rm Id}$, we can assume that the $C^\infty$ jets of $g_1$ and $g_2$ coincide at $\partial M$ and the same is true for the geodesic vector fields $X_{g_1}$ and $X_{g_2}$ (as vector fields on $TM$). We denote by $S_{i}M$ the unit tangent bundle of $(M,g_i)$. Consider an arbitrary point $y\in\partial_0 SM\cup \partial_- SM$. Up to switching the roles of $g_1$ and $g_2$, we can assume that $\tau_{g_1}^+(y)\geq\tau_{g_2}^+(y)$. If $\tau_{g_1}^+(y)=0$, then $\tau_{g_2}^+(y)=0$ and $\sigma_{g_1}(y)=\sigma_{g_2}(y)=y$. Let us now consider the case where $\tau:=\tau_{g_1}^+(y)>0$, so that we have a forward geodesic $\gamma_1:[0,\tau]\to S_{1}M$, $\gamma_1(t):=\varphi^{g_1}_t(y)$. We define the closed subset $F:=\big\{t\in[0,\tau]\ \big|\ \gamma_1(t)\in\partial SM\big\}$, and a curve $\gamma_2:[0,\tau]\to S_{2}M$ as follows. On $F$, we set $\gamma_2|_{F}:=\gamma_1|_{F}$. Since $X_{g_1}=X_{g_2}$ on $\pl SM$, for each connected component $J$ of the interior $F^\circ$, we have $\gamma_2(t)=\varphi_{t-\inf J}^{g_2}(\gamma_2(\inf J))$ for all $t\in J$. For each connected component $I$ of $[0,\tau]\setminus F$, one has $\gamma_1|_I\subset S_{1}M^\circ$, and we set $\gamma_2(t):=\varphi_{t-\inf I}^{g_2}(\gamma_1(\inf I))$ for each $t\in I$; since $(t^+_{g_1},s_{g_1})=(t^+_{g_2},s_{g_2})$, we have that $\gamma_2(\sup I)=\gamma_1(\sup I)$.  Summing up, we constructed a continuous curve $\gamma_2:[0,\tau]\to S_{2}M$ that restricts to a smooth solution of the equation $\dot\gamma_2=X_{g_2}\circ\gamma_2$ on the dense subset $[0,\tau]\setminus \partial F$. We claim that $\gamma_2\in C^1$ and satisfies this equation on $[0,\tau]$. Indeed, since the $C^\infty$ jets of $X_{g_1}$ and $X_{g_2}$ coincide at $\partial SM$, we have that $d(\varphi_t^{g_1}(z),\varphi_t^{g_2}(z))= \mc O(t^\infty)$ uniformly in $z\in\partial SM$ as $t\to0$, where $d$ is any Riemannian distance on $TM$. This, together with $\gamma_1|_{F}=\gamma_2|_{F}$, implies that $d(\gamma_1(t),\gamma_2(t))=\mc O\big(\inf\big\{|t-t'|\ \big|\ t'\in F\big\}^\infty\big)$ uniformly in $t\in[0,\tau]$. Therefore, for each $t\in\partial F$ and $t'\in[0,\tau]$ sufficiently close to $t$, we have the following expression in local coordinates around $\gamma_1(t)=\gamma_2(t)$:
\begin{align*}
\gamma_2(t')-\gamma_2(t)=& \gamma_1(t')-\gamma_1(t)+\mc{O}(|t'-t|^\infty)
=(t'-t) X_1(\gamma_1(t))+\mc{O}((t'-t)^2).
\end{align*}
In particular, $\gamma_2$ is differentiable at $t$ with derivative $\dot\gamma_2(t)=X_1(\gamma_1(t))=X_2(\gamma_2(t))$. We conclude that $\gamma_2:[0,\tau]\to S_{2}M$ is a smooth solution of  $\dot\gamma_2=X_{g_2}\circ\gamma_2$, and therefore it is a geodesic $\gamma_2(t)=\varphi_t^{g_2}(y)$ satisfying $\dot\gamma_2(\tau)=\dot\gamma_1(\tau)=\sigma_{g_1}(y)$. This proves that $\tau_{g_2}^+(y)\geq \tau$, and therefore $\tau_{g_2}^+(y)= \tau_{g_1}^+(y)=\tau$ and  $\sigma_{g_2}(y)= \sigma_{g_1}(y)$.
\end{proof}
As for the equivalence of the boundary distance function $\beta_g$ defined by \eqref{betag} and the lens data, we are going to prove the following statement for surfaces. 

\begin{thm}\label{betadetS}
Let $M$ be a connected, simply connected, compact surface with boundary. If $g_1$ and $g_2$ are two Riemannian metrics without conjugate points such that $\beta_{g_1}=\beta_{g_2}$, then $(\tau^+_{g_1}, \sigma_{g_1})=(\tau^+_{g_2},\sigma_{g_2})$.
\end{thm}
\begin{proof}
First, we note that $\beta_{g_1}=\beta_{g_2}$ implies that $g_1|_{T\pl M}=g_2|_{T\pl M}$:
indeed, for each $x,x'\in \pl M$ close one to each other and $\alpha:[0,1]\to \pl M$ the shortest curve between $x$ and $x'$ in $\pl M$, one has 
\begin{equation}\label{inegalpha}
\ell_{g_j}(\alpha)\geq \sup\left\{\left. \sum_{i=1}^{N-1}\beta_{g_j}(\alpha(t_i),\alpha(t_{i+1}))\, \right|\, 
N\in \nn, 0=t_1<t_2<\dots<t_N=1\right\}
\end{equation}
but, if $(M,\til{g}_j)$ is an extension of $(M,g_j)$, the right hand side is greater than 
\[\sup\left\{\left. \sum_{i=1}^{N-1}d_{\til{g}_j}(\alpha(t_i),\alpha(t_{i+1}))\, \right |\, 
N\in \nn, 0=t_1<t_2<\dots<t_N=1\right\}=\ell_{\til{g}_i}(\alpha)=\ell_{g_i}(\alpha).\]
thus \eqref{inegalpha} is an equality. Once we know the length of curves on $\pl M$, we know $g_i|_{T\pl M}$.
We break the proof  Theorem~\ref{betadetS} into several lemmas.
\begin{lem}
\label{l:approximating_first_hit}
Let $(M,g)$ be a simply connected, compact Riemannian surface with boundary and without conjugate points. For each $y\in \partial_0SM\cup\partial_-SM$ satisfying $t^+_g(y)>0$ there exists a sequence $\{y_n\,|\,n\in\N\} \subset \partial_- SM\cap s_g^{-1}(\partial_+ SM)$ such that $y_n\to y$,  $t^+_g(y_n) \to t^+_g(y)$, and $s_g(y_n) \to s_g(y)$.
\end{lem}

\begin{proof}
By Lemma~\ref{l:no_conjugate_points}, we consider $(M,g)$ to be contained in the interior of a simply connected compact surface with boundary $(M_1,g)$ without conjugate points. Let $y=(x,v)$ be a point as in the statement, and set $y'=(x',v'):=s_g(x,v)$ and $T:=t^+_g(y)>0$. The cases where $y\in\partial_-SM$ or $y'\in\partial_+SM$ are easy, and therefore we focus on the remaining case in which $y,y'\in\partial_0SM$. For a sufficiently small $\Theta>0$, we have a well defined geodesic $\zeta:(-\Theta,0]\to M_1$, $\zeta(\theta):=\exp_{x'}(\theta\,\nu_{x'})$, and a  smooth map
\begin{align*}
\psi:[0,T]\times[-\Theta,0]\to M_1,
\qquad
\psi(t,\theta)=\exp_{x}\big( t \,T^{-1} \exp_{x}^{-1}(\zeta(\theta)) \big),
\end{align*}
\begin{figure}
\begin{center}
\begin{footnotesize}
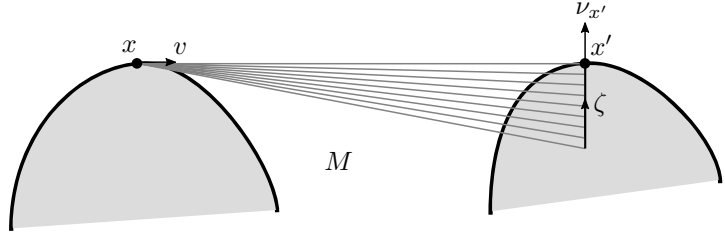 
\end{footnotesize}
\caption{The region covered by the map $\psi$, foliated by geodesics $t\mapsto\psi(t,\theta)$.}
\label{f:coordinates}
\end{center}
\end{figure}%
see Figure~\ref{f:coordinates}. Since there are no conjugate points in the extended manifold $(M_1,g)$, the map $\psi$ restricts to a diffeomorphism onto its image on  $(0,T]\times(-\Theta,0]$. Since $\psi(t,0)=\exp_{x}(tv)\in M^\circ$ for all $t\in(0,T)$, for each $\epsilon\in(0,T/2)$ there exists $\delta_\epsilon\in(0,\Theta)$ such that 
\begin{align*}
\psi(t,\theta)\in M^\circ,
\qquad
\forall t\in[\epsilon,T-\epsilon],\ \theta\in (-\delta_\epsilon,0].
\end{align*}
Consider the open subset 
\[
J_\epsilon:=\big\{t\in(0,T)\ \big|\ \psi(t,\theta)\in\partial M \mbox{ for some }\theta\in(-\delta_\epsilon,0)\big\}
\subset
(0,\epsilon)\cup(T-\epsilon,T).
\]
Since $y,y'\in\partial_0 SM$, we can apply the implicit function theorem: for $\epsilon_0>0$ sufficiently small, there exists a smooth map $\theta_0:J_{\epsilon_0}\to(-\delta_{\epsilon_0},0)$ such that $\psi(t,\theta)\in\partial M$ for some $(t,\theta)\in J_{\epsilon_0}\times(-\delta_{\epsilon_0},0)$ if and only if $\theta=\theta_0(t)$. Let $R\subset (-\delta_{\epsilon_0},0)$ be the set of regular values of $\theta_0$. Such an $R$ has full measure in $(-\delta_{\epsilon_0},0)$  according to Sard's Theorem. In particular, we can find a sequence $\{\theta_n\,|\,n\in\N\}\subset R$ such that $\theta_n\to0$. We set 
\begin{align*}
a_n :=\max\big\{t\in[0,\epsilon_0]\ \big|\ \psi(t,\theta_n)\in\partial M\big\},\quad 
b_n :=\min\big\{t\in[T-\epsilon_0,T]\ \big|\ \psi(t,\theta_n)\in\partial M\big\},
\end{align*}
and consider the geodesic $\gamma_n:[a_n,b_n]\to M$ given by $\gamma_n(t):=\psi(t,\theta_n)$. Notice that $\gamma_n(t)\in M^\circ$ for all $t\in(a_n,b_n)$. If $a_n=0$ (as is the case when $\nu_x$ and $\nu_{x'}$ point to different sides of the geodesic joining $x$ and $x'$), the tangent vectors  $\dot\gamma_n(a_n)$ and $v$ are linearly independent, and in particular $\dot\gamma_n(a_n)$ is transverse to $\partial M$; if $a_n>0$, since $\theta_n\in R$, we still infer that $\dot\gamma_n(a_n)$ is transverse to $\partial M$. Analogously, $\dot\gamma_n(b_n)$ is transverse to $\partial M$. Finally, since $\theta_n\to0$, we have $a_n\to 0$ and $y_n:=(\gamma_n(a_n),\dot\gamma_n(a_n)/\|\dot\gamma_n(a_n)\|_g)\to y$. Such $y_n$'s have the desired properties.
\end{proof}

\begin{lem}
\label{l:same_transverse_lens}
Under the assumptions of Theorem~\ref{betadetS}, we have $s_{g_1}(y)=s_{g_2}(y)$ and $t^+_{g_1}(y)=t^+_{g_2}(y)$ for each $y\in\partial_-SM\cap s_{g_1}^{-1}(\partial_+SM)$.
\end{lem}

\begin{proof}
We will simply write $\beta:=\beta_{g_1}=\beta_{g_2}$. Consider an arbitrary point $y=(x,v)\in\partial_-SM\cap s_{g_1}^{-1}(\partial_+SM)$, and set $y'=(x',v'):=s_{g_1}(x,v)$. Notice that the geodesic $\gamma_1:[0,t^+_{g_1}(y)]\to M$, $\gamma_1(t):=\exp_x^{g_1}(tv)$ intersects the boundary $\partial M$ transversely at $t=0$ and $t=t^+_{g_1}(y)$, while $\gamma_1(t)\in M^\circ$ for all $t\in(0,t^+_{g_1}(y))$. This implies that
\begin{align}\label{e:triangular_inequality}
\beta(x,x')<\beta(x,x'')+\beta(x'',x'),
\qquad
\forall x''\in\partial M\setminus\{x,x'\}.
\end{align}
By Proposition~\ref{p:unique_minimizer}, $\gamma_1$ is the only length minimizing curve joining $x$ and $x'$, and in particular $t_{g_1}^+(y)=\beta(x,x')$.
A standard computation implies that $\beta$ is smooth near $(x,x')$, and
\begin{align*}
v = -\nabla_x\beta(x,x') + \sqrt{1-\|\nabla_x\beta(x,x')\|^2_{g_1}}\nu_x,\quad 
v' = \nabla_{x'}\beta(x,x') - \sqrt{1-\|\nabla_{x'}\beta(x,x')\|_{g_1}^2}\nu_{x'},
\end{align*}
where $\nabla_x$ and $\nabla_{x'}$ denote the Riemannian gradients with respect to $x$ and $x'$ respectively (recall that $g_1=g_2$ on $\pl M$). Let $\gamma_2:[0,\beta(x,x')]\to M$ be a $C^1$-curve such that $\|\dot\gamma_2\|_{g_2}\equiv1$, $\gamma_2(0)=x$, $\gamma_2(\beta(x,x'))=x'$ (namely, $\gamma_2$ is a length minimizing curve for $g_2$ joining $x$ and $x'$). The strict triangular inequality~\eqref{e:triangular_inequality} implies that $\gamma_2(t)\in M^\circ$ for all $t\in(0,\beta(x,x'))$, and therefore it is a geodesic for $g_2$. We set $z:=(\gamma_2(0),\dot\gamma_2(0))$. A priori, $\gamma_2$ may intersect the boundary $\partial M$ tangentially at the endpoints, but nevertheless Lemma~\ref{l:approximating_first_hit} implies that there exists a sequence $\{z_n=(x_n,v_n)\} \subset \partial_- SM\cap s_{g_2}^{-1}(\partial_+ SM)$ such that $z_n\to z$,  $t^+_{g_2}(z_n) \to t^+_{g_2}(z)$, and $(x_n',v_n'):=s_{g_2}(z_n) \to s_{g_2}(z)$. As above, we have
\begin{align*}
v_n & = -\nabla_{x_n}\beta(x_n,x'_n) + \sqrt{1-\|\nabla_{x_n}\beta(x_n,x'_n)\|^2_{g_1}}\nu_{x_n},\\
v'_n & = \nabla_{x'_n}\beta(x_n,x'_n) - \sqrt{1-\|\nabla_{x'_n}\beta(x_n,x'_n)\|_{g_1}^2}\nu_{x'_n},
\end{align*}
and therefore, since $\beta$ is smooth in a neighborhood of $(x,x')$,
\begin{align*}
\dot\gamma_2(0)=\lim_{n\to\infty} v_n = v,
\quad \dot\gamma_2(\beta(x,x'))=\lim_{n\to\infty} v_n' = v'.
\end{align*}
This shows that $s_{g_1}(y)=s_{g_2}(y)$ and  $t^+_{g_1}(y)=t^+_{g_2}(y)=\beta(x,x')$.
\end{proof}

\begin{lem}\label{l:positive_t}
Under the assumptions of Proposition~\ref{betadetS}, we have $s_{g_1}(y)=s_{g_2}(y)$ and $t_{g_1}^+(y)=t_{g_2}^+(y)$ for each $y\in\partial SM$ such that $t_{g_1}^+(y)>0$.
\end{lem}

\begin{proof}
Consider an arbitrary point $y=(x,v)\in\partial SM$ such that $t^+_{g_1}(y)>0$. We set $y'=(x',v'):=s_{g_1}(y)$, and notice that we have the strict triangular inequality
\begin{align}\label{e:strict_triangular_inequality_again}
\beta(x,x')<\beta(x,x'')+\beta(x'',x'),\qquad
\forall x''\in M\setminus\{x,x'\},
\end{align}
where $\beta=\beta_{g_1}=\beta_{g_2}$.
The points $y$ and $y'$ may belong to the tangential boundary $\partial_0SM$, but nevertheless Lemma~\ref{l:approximating_first_hit} implies that there exists a sequence $\{y_n\,|\, n\in\N\} \subset \partial_- SM\cap s_{g_1}^{-1}(\partial_+ SM)$ such that $y_n\to y$,  $t^+_{g_1}(y_n) \to t^+_{g_1}(y)$, and $s_{g_1}(y_n) \to s_{g_1}(y)$. Lemma~\ref{l:same_transverse_lens} implies that $s_{g_1}(y_n)=s_{g_2}(y_n)$ and $t^+_{g_1}(y_n)=t^+_{g_2}(y_n)$. Therefore, there is a limit $g_2$-geodesic $\gamma:[0,t^+_{g_1}(y)]\to M$, $\gamma(t)=\exp_x^{g_2}(tv)$ such that  $\gamma(0)=x$ and $\gamma(t^+_{g_1}(y))=x'$. Notice that the $g_2$-length of $\gamma$ is precisely $\beta(x,x')$. Therefore, the inequality~\eqref{e:strict_triangular_inequality_again} implies that $\gamma(t)\in M^\circ$ for all $t\in(0,t^+_{g_1}(y))$. This shows that $s_{g_1}(y)=s_{g_2}(y)$ and $t^+_{g_1}(y)=t^+_{g_2}(y)=\beta(x,x')$.
\end{proof}

We can now finish the proof of Theorem~\ref{betadetS}.
Let $y\in\partial_0SM\cup\partial_-SM$. If $t^+_{g_1}(y)>0$, Lemma~\ref{l:positive_t} implies that $s_{g_1}(y)=s_{g_2}(y)$ and $t^+_{g_1}(y)=t^+_{g_2}(y)$. If instead $t^+_{g_1}(y)=0$, we claim that $t^+_{g_2}(y)=0$ as well (and thus $s_{g_1}(y)=s_{g_2}(y)=y$ and $t^+_{g_1}(y)=t^+_{g_2}(y)=0$); indeed, if $t^+_{g_2}(y)>0$, we could apply Lemma~\ref{l:positive_t} switching the roles of $g_1$ and $g_2$, and we would obtain the contradiction $t^+_{g_1}(y)=t^+_{g_2}(y)>0$. It suffices to use Proposition \ref{equivalencescat} to conclude.
\end{proof}

\subsection{Determination of the conformal factor: proof of Theorems~\ref{th1} and~\ref{th2bis}}\label{ss:final}

The following is a well known consequence of Santalo's formula.

\begin{prop}\label{p:same_volume}
Let $(M_i,g_i)$, $i=1,2$, be compact Riemannian manifolds with boundary such that  $(\pl M_1,g_1|_{T\pl M_1})=(\pl M_2,g_2|_{T\pl M_2})$. Assume that 
the trapped set $\mc{K}_i$ has Liouville measure $0$ and that $\tau_{g_1}^+=\tau_{g_2}^+$, where $\tau_{g_i}^+$ is the forward exit time function of $(M,g_i)$. Then $\Vol(M,g_1)=\Vol(M,g_2)$.
\end{prop}

\begin{proof}
It suffices to observe that the measure $\mu_\nu$ on $\partial SM_i$ is determined by $g_i$ at $\pl M_i$ and then use Santalo formula (Lemma~\ref{santalofor})
\begin{align*}
 \Vol(M_i,g_i)=
 \frac{1}{\Vol(S^{n-1})}
 \int_{\partial_{\mathrm{-}}SM_i} \tau_{g_i}^+(x,v)\,d\mu_\nu
\end{align*}
to see that the volume is determined by $\tau_{g_i}^+$.
\end{proof}

By combining the results in this section with Theorem~\ref{th2} and an argument due to Croke~\cite[Theorem~C]{Croke:1991wc}, we can finally provide a proof of Theorem~\ref{th1}.

\begin{proof}[Proof of Theorem~\ref{th1}]
As in the proof of Proposition~\ref{betadetS}, we have $g_1|_{T\pl M}=g_2|_{T\pl M}$ and by Stefanov-Uhlmann \cite[Theorem 1]{Stefanov:2009lp}, up to pulling back $g_2$ by a diffeomorphism, the $C^\infty$-jet of $g_1$ and $g_2$ agree at $\pl M$. Next, by
Theorem~\ref{th2}, we can reduce to the case where $g_1$ and $g_2$ are conformal, i.e $g_1=e^{2\rho}g_2$ for some $\rho\in C^\infty(M)$ such that $\rho|_{\pl M}\equiv 0$, and  then $d\mu_{g_1}=e^{2\rho}d\mu_{g_2}$.
We denote by $\mu_{i}$ the Liouville measure on the unit tangent bundle $S_{i}M$ induced by the Sasaki metric $G_i$ of $g_i$. As in Proposition~\ref{p:same_volume}, $\mu_\nu$ on $\pl_-S_{1}M=\pl_- S_{2}M$ agrees for $g_1$ and $g_2$. By Theorem~\ref{th2}, we can reduce to the case where $g_1$ and $g_2$ are conformal, i.e $g_1=e^{2\rho}g_2$ for some $\rho\in C^\infty(M)$ such that $\rho|_{\pl M}\equiv 0$, and  then $d\mu_{g_1}=e^{2\rho}d\mu_{g_2}$.
By Proposition~\ref{p:same_volume}, $\Vol(M,g_1)=\Vol(M,g_2)$, and the H\"older inequality
\begin{align*}
\| e^\rho \|_{L^1(M,g_2)}
\leq
\| e^\rho \|_{L^2(M,g_2)}\| 1 \|_{L^2(M,g_2)}
=
\Vol(M,g_1)^{1/2}\Vol(M,g_2)^{1/2}=\Vol(M,g_2).
\end{align*}
is strict unless $\rho=0$. 
By Lemma~\ref{l:scs_non_trapping}, both $(M,g_1)$ and $(M,g_2)$ are non-trapping. Let $\varphi^{g_i}_t$ denotes the geodesic flow for $g_i$ on $S_{i}M$. By Proposition~\ref{p:unique_minimizer}, each geodesic of $(M,g_i)$ joining given points is the unique curve of minimal $g_i$-length. In particular 
\begin{align*}
\int_0^{\tau_{g_2}^+(x,v)} \|\varphi_t^{g_2}(x,v)\|_{g_1}d t
\geq
\tau_{g_1}^+(x,v),
\qquad 
\forall (x,v)\in \partial_{-}S_2M.
\end{align*}
This, together with Santalo's formula (Lemma~\ref{santalofor}), implies
\begin{align*}
\| e^\rho \|_{L^1(M,g_2)}
=&
\frac{1}{2\pi}
\int_{S_{2}M}
\|v\|_{g_1} d\mu_{G_2}(x,v)
 =
\frac{1}{2\pi}
\int_{\partial_{-}S_{2}M}
\int_0^{\tau_{g_2}^+(x,v)}\|\varphi_t^{g_2}(x,v)\|_{g_1}d t\, d\mu_\nu(x,v)\\
&\geq
\frac{1}{2\pi}
\int_{\partial_-S_{2}M}
\tau_{g_1}^+(x,v)\, d\mu_\nu(x,v)
=
\frac{1}{2\pi}
\int_{\partial_-S_{2}M}
\tau_{g_2}^+(x,v)\, d\mu_\nu(x,v)\\
&\geq \Vol(M,g_2).
\end{align*}
Therefore, the above H\"older inequality is satisfied as an equality, and $\rho\equiv0$.
\end{proof}

We now sketch the argument of Zhou~\cite[Section~4]{Zhou:2011wq} for deriving Theorem~\ref{th2bis} from Theorem~\ref{th2}.

\begin{proof}[Proof of Theorem~\ref{th2bis} by X. Zhou]
Since $(M_1,g_1)$ is non-trapping, the quantity $\tau:=\max \tau_{g_1}^+$ is finite. There exists a finite covering space $\pi:M_1'\to M_1$ such that the shortest non-contractible piecewise smooth loop in $M_1'$ has length with respect to $g_1':=\pi^*g_1$ larger than~$2\tau$, see \cite[Proposition 4.2.2]{Zhou:2011wq}. We are going to prove  Theorem~\ref{th2bis} with $(M_1,g_1)$ replaced by $(M_1',g_1')$. By a result of Croke \cite[Theorem~1.2]{Croke:2005bh}, this will imply Theorem~\ref{th2bis} for $(M_1,g_1)$ as well.

Let $(M_2',g_2')$ be a connected, non-trapping, oriented compact Riemannian surface with boundary, without conjugate points, and with the same lens data as $(M_1',g_1')$. By means of Theorem~\ref{th2}, we can assume without loss of generality that $M_1'=M_2'=:M$ and $g_2'=e^{2\rho}g_1'$ for some $\rho\in C^\infty(M)$ with $\rho|_{\partial M}\equiv0$. For each $(x,v)\in\partial_-SM$, let $\gamma_{i,x,v}:[0,\tau_{g_i'}^+(x,v)]\to M$ be the $g_i'$-geodesic such that $\gamma_{i,x,v}(0)=x$ and $\dot\gamma_{i,x,v}(0)=v$. Notice that \[\tau':=\tau_{g_2'}^+(x,v)=\tau_{g_1'}^+(x,v)\leq\tau.\] We claim that 
\begin{align}
\label{e:comparing_lengths}
\int_0^{\tau'} \|\dot\gamma_{2,x,v}(t)\|_{g_1'} dt
\geq
\tau'.
\end{align}
Indeed, $\gamma_{1,x,v}(\tau')=\gamma_{2,x,v}(\tau')$. If $\gamma_{1,x,v}$ and $\gamma_{2,x,v}$ are not homotopic with fixed endpoints, they can be joined to form a non-contractible loop $\gamma_{1,x,v}*\overline{\gamma_{2,x,v}}$ in $M$, which thus has $g_1'$-length larger than $2\tau$; since $\gamma_{1,x,v}$ has $g_1'$-length smaller than $\tau$, this implies 
\begin{align*}
\int_0^{\tau'} \|\dot\gamma_{2,x,v}(t)\|_{g_1'} dt
\geq
2\tau
-
\int_0^{\tau'} \|\dot\gamma_{1,x,v}(t)\|_{g_1'} dt
\geq
\tau
\geq
\tau'.
\end{align*}
On the other hand, if $\gamma_{1,x,v}$ and $\gamma_{2,x,v}$ are homotopic paths, since $(M,g_1')$ is without conjugate points we can apply Proposition~\ref{p:unique_minimizer}, which implies that $\gamma_{1,x,v}$ has minimal $g_1'$-length among the paths in his homotopy class, and once again~\eqref{e:comparing_lengths} follows.
Now that we have established~\eqref{e:comparing_lengths} for all $(x,v)\in\partial_-SM$, we can carry over Croke's argument as in the proof of Theorem~\ref{th1} and conclude that $\rho\equiv0$.
\end{proof}

In a compact Riemannian manifold $(M,g)$, a point $x\in\partial M$ is called a \emph{switch point} for $g$ when the curvature of $\partial M$ with respect of $g$ vanishes at $x$, but does not vanish identically on a neighborhood of it. In the setting of Theorem~\ref{th2bis}, if we further assume that $\partial M$ has finitely many switch points for $g_1$, a recent result of Croke and Wen \cite{Croke:2015qy} implies that the hitting times $t_{g_1}^+$ and $t_{g_2}^+$ coincide provided the scattering maps $s_{g_1}$ and $s_{g_2}$ coincide. By combining this result with our Theorem~\ref{th2bis} and Propositions~\ref{equivalencescat} and~\ref{betadetS}, we obtain the following scattering rigidity result.

\begin{cor}\label{c:croke}
Let $(M_1,g_1)$ and $(M_2,g_2)$ be two non-trapping, oriented compact Riemannian surfaces with boundary, without conjugate points, with the same hitting scattering maps, and such that $\partial M_1$ has finitely many switch points for $g_1$. Then there is a diffeomorphism $\psi:M_1\to M_2$  such that $\psi^*g_2=g_1$.
\hfill\qed
\end{cor}

\bibliography{_biblio}
\bibliographystyle{amsalpha}

\end{document}